\documentclass[12pt]{amsart}
\usepackage[utf8]{inputenc}
\usepackage{amsxtra,amssymb,amsthm,amsmath,amscd,mathrsfs, epsfig, eufrak}
\usepackage{amscd, amsmath, mathrsfs, amssymb, amsthm, amsxtra, bbding, epsfig, eucal, eufrak, graphicx, latexsym, mathrsfs}
\usepackage[all]{xy}
\usepackage{tikz-cd} 
\usepackage{tikz}
\usepackage[normalem]{ulem}
\usepackage{soul}
\usepackage{color}

\setstcolor{red}

\setstcolor{blue}
\makeatletter
\@namedef{subjclassname@2010}{%
 \textup{2010} Mathematics Subject Classification}
\makeatother

\def \A {{\mathbb A}}

\def \C {{\mathbb C}}

\def \N {{\mathbb N}}
\def \P {{\mathbb P}}
\def \Q {{\mathbb Q}}
\def \R {{\mathbb R}}

\def \Z {{\mathbb Z}}
\def \OO {{\mathcal O}}

\def \d {\,{\rm d}}
\def\re{{\Re e\,}}

\def\Bl{\hbox{{\rm Bl}}}

\def\leq{\leqslant}
\def\geq{\geqslant}
\def\le{\leqslant}
\def\ge{\geqslant}

\theoremstyle{plain}
\newtheorem{theorem}{Theorem}[section]
\newtheorem{proposition}{Proposition}[section]
\newtheorem{lemma}[proposition]{Lemma}

\theoremstyle{remark}

\numberwithin{equation}{section}
\addtocounter{footnote}{1}

\begin{document}

\title[Manin's conjecture for a class of singular cubic hypersurfaces]
{Manin's conjecture for a class of \\ singular cubic hypersurfaces}
\author{Jianya Liu, Jie Wu \& Yongqiang Zhao}

\address{%
Jianya Liu
\\
School of Mathematics
\\
Shandong University
\\
Jinan
\\
Shandong 250100
\\
China} \email{jyliu@sdu.edu.cn}

\address{%
Jie Wu
\\
CNRS\\
Institut \'Elie Cartan de Lorraine\\
UMR 7502\\
54506 Van\-d\oe uvre-l\`es-Nancy\\
France}
\curraddr{%
Universit\'e de Lorraine\\
Institut \'Elie Cartan de Lorraine\\
UMR 7502\\
54506 Van\-d\oe uvre-l\`es-Nancy\\
France
}
\email{jie.wu@univ-lorraine.fr}

\address{%
Yongqiang Zhao
\\
Westlake Institute for Advanced Study
\\
Shilongshan Road, Cloud Town, Xihu District
\\
Hangzhou
\\
Zhejiang Province, 310024    
\\
China} 
\email{yzhao@wias.org.cn}

\date{\today}

\begin{abstract} 
Let $n$ be a positive multiple of $4$.
We establish an asymptotic formula for the number of rational points of bounded height on 
singular cubic hypersurfaces $S_n$ defined by 
$$
x^3=(y_1^2 + \cdots + y_n^2)z .
$$  
This result is new in two aspects: first, it can be viewed as a modest start on the study of density 
of rational points 
on those singular cubic hypersurfaces which are not covered by the classical theorems 
of Davenport or Heath-Brown; second, it proves Manin's  conjecture  for singular 
cubic hypersurfaces $S_n$ defined above.  

\end{abstract}

\subjclass[2000]{11D45, 11N37}
\keywords{Cubic hypersurface; Manin's conjecture; rational point; asymptotic formula. }   

\maketitle   	
   	
\section{Introduction}
\subsection{The result} 
The aim of the paper is to study the density of 
rational points on the cubic hypersurfaces
$S_n$ defined by  
\begin{equation}\label{def:Sn}
x^3=(y_1^2+\cdots+y_{n}^2)z, 
\end{equation}
where $n\geq 3$ is an integer. 
It is well-known that for any $S_n$ with $n\geq 3$,   
the heuristic of the circle method does not apply, since there are too many 
solutions with $x=z=0$. One 
therefore counts such solutions of  \eqref{def:Sn} that neither $x$ nor 
$z$ vanishes.  

If a point in $\P^{n+1}$  is represented by $(x, y_1, \dots, y_n, z) \in \Z^{n+2} $ with
coprime coordinates, then 
\begin{equation}\label{def:Height}
H(x: y_1: \ldots : y_n : z) = \max\{|x|, \textstyle\sqrt{y_1^2+\cdots+y_n^2}, |z| \}^{n-1} 
\end{equation} 
is a natural anticanonical height function on $S_n(\Q)$.  
Let $N_n(B)$  denote the number of rational points on
\eqref{def:Sn} satisfying 
\begin{equation}\label{def:Box}
H(x: y_1: \ldots : y_n : z)\leq B, \;
x\not=0, \;
z\not=0. 
\end{equation}
In the classical setting of counting of integral solutions of \eqref{def:Sn} 
by the circle method,  
one usually counts solutions without the coprime condition. We therefore let 
\begin{equation}\label{def:Affineheight}
H^*(x, y_1, y_2, \ldots , y_n, z) = \max\{|x|, \textstyle\sqrt{y_1^2+\cdots+y_n^2}, |z| \}
\end{equation} 
for any point $(x, y_1, y_2, \ldots , y_n, z)\in {\Z}^{n+2}$, 
and we denote accordingly  by $N^*_n(B)$ the number of integral solutions of 
\eqref{def:Sn} satisfying 
\begin{equation}\label{def:Box*}
H^*(x, y_1, \ldots , y_n , z)\leq B, \;
x\not=0, \;
z\not=0.
\end{equation}
These two height functions $H$ and $H^*$ are closely related. 
The purpose of this paper is to establish an asymptotic formula for $N_n^*(B)$, 
and hence deduce an asymptotic formula for $N_n(B)$, 
as $B\to \infty$.  

One sees easily that the above $H$ and $H^*$  
are induced by the norm  
$\| \cdot \|: {\Bbb R}^{n+2} \to {\Bbb R}_{\geq 0}$ defined as  
$\|(x, y_1, \ldots, y_n, z)\|= \max\{|x|, \sqrt{y_1^2+\cdots+y_n^2}, |z|\}$. 
Of course it is possible to use height functions other than $H$ or $H^*$, 
but it turns out that 
these specific $H$ and $H^*$ are more natural. One observes that, 
in the affine space $\A^{n+2}$, counting integral solutions of 
\eqref{def:Sn} with bounded height $H^*(P)\leq B$ 
is equivalent to counting  points inside the poly-cylinder  
$[-B, B]^2 * \text{\rm Ball}(0, B)$. 

The main results of the paper are asymptotic formulae for $N_n(B)$ and 
$N^*_n(B)$ when $n$ is a multiple of $4$.  
Our method works well for all integers $n\geq 3$, 
but for ease of presentation we just focus on the case when is $n$ is a multiple 
of $4$, and leave the general case to another occasion.  

For ease of presentation, we will give detailed proof of the following 
Theorem~\ref{thm1} which corresponds to  the typical case $n=4$. 
The general case $n=4k$ can be 
treated in the same way, and only slight modifications are necessary;  see 
Theorem~\ref{thm5} and its proof in \S\ref{GeneralCase: Sn}. 

\begin{theorem}\label{thm1} 
As $B\to\infty$, we have 
\begin{equation}\label{eq:N4B}
N_4(B)
= \mathcal{C}_4 B(\log B)^2 \, \bigg\{1+O\bigg(\frac{1}{\log B}\bigg)\bigg\} 
\end{equation}
and 
\begin{equation}\label{eq:N4*B}
N^*_4(B)
= \mathcal{C}^*_4 B^3(\log B)^2 \, \bigg\{1+O\bigg(\frac{1}{\log B}\bigg)\bigg\}, 
\end{equation}
where
\begin{equation}\label{def:CS4}
\mathcal{C}_4 := \frac{\mathcal{C}^*_4}{9\zeta(3)},  
\qquad
\mathcal{C}^*_4 := \frac{16}{3} \mathscr{C}_4, 
\end{equation}
and $\mathscr{C}_4$ is the positive constant defined as in \eqref{def:a2}, and $\zeta$ is the 
Riemann zeta-function. 
\end{theorem}

Theorem~\ref{thm1} can be viewed in two perspectives, the first of which 
is cubic forms representing zero, and the second is Manin's conjecture. 

\subsection{Cubic forms}  
We put Theorem~\ref{thm1} in the perspective of 
cubic forms representing zero. Let $C(x_1, \dots, x_s)\in {\Bbb Z}[x_1, \dots, x_s]$ 
be a cubic form of $s$ variables. 
Then Davenport \cite{Dav} showed that there exits a non-zero integral vector ${\bf x}$ 
such that $C({\bf x})=0$ provided that $s$ is at least $16.$  This $16$ is reduced to 
$14$ by Heath-Brown \cite{HeaBro}. In \cite{Dav} and \cite{HeaBro}, two alternative 
cases have been considered separately.  

To state the first alternative, we need to introduce a Geometric Condition of 
Davenport, in terms of the Hessian 
$$
\mathrm{Hess}(C)=\bigg (\frac{\partial ^2 C}{\partial x_i \partial x_j} \bigg) 
$$
of a given cubic form $C=C({\bf x})$. We remark that if one writes $C({\bf x})$ in the  form 
$$
C(x_1, \dots, x_s)=\sum_{i,j,k}c_{ijk}x_ix_jx_k 
$$
so that the coefficients $c_{ijk}$ are symmetric in the indices $i, j, k$, then up to a constant
$\mathrm{Hess} (C)$ is equal to the following matrix 
$$
M({\bf x})=\bigg(\sum_{k}c_{ijk}x_k\bigg),
$$
which is used in the work of Davenport \cite{Dav} and Heath-Brown \cite {HeaBro}. 

\medskip
\noindent 
{\sc Geometric Condition (G).} \textit{The estimate 
\begin{equation} \label{G}
\# \{{\bf x}\in  {\Bbb Z}^s: |{\bf x}|\leq B, \, \mathrm{rank}({\mathrm{Hess}(C)})=r\}\ll_\varepsilon B^{r+\varepsilon}   
\end{equation} 
holds for all nonnegative integers $r\leq s$. Here $|{\bf x}|\leq B$ means that 
each coordinate $x_j$ of ${\bf x}$ satisfies $|x_j|\leq B$. } 

\medskip

In the first alternative,  Davenport and Heath-Brown established an asymptotic formula for the number of solutions of  cubic forms $C({\bf x})$ satisfying the Geometric Condition (G), 
and  with $s\geq 16$ and $s \geq 14$, respectively.  

While in the second alternative,  i.e. if \eqref{G} 
fails for some nonnegative integer $r\leq s$, it is only proved that 
the form $C({\bf x})$ has at least one non-trivial zero for geometric reasons, and therefore  it 
leaves the question of establishing  an asymptotic formula for the number of its 
zeros untouched. In addition to the desire for a complete theory in the second alternative,   
a study of rational points on cubic hypersurfaces in the second alternative will also 
supply a flourishing testing ground for general versions of Manin's conjecture. This 
explains the motivation of this paper.

Now let us take a closer look at the Geometric Condition (G).   
From the analytic point of view, 
it is a sufficient condition to guarantee the desired cancellation from the exponential sum 
$$
\sum_{|{\bf x}|\leq B}\exp\big(2\pi \mathrm{i} C({\bf x})\alpha\big) 
$$
in the circle method;  while from the geometric point of view, in some sense, 
it is a quantitative  measure of the largest 
possible dimension of  all linear subspaces that can be embedded into 
the hyperspace $C({\bf x})=0$.  For example, 
we have the following result in an extremal case.  
 
 \begin{lemma}
Let $C({\bf x})$ be a cubic form with $s\geq 6$ variables, and suppose that 
the cubic hypersurface $C({\bf x})=0$ contains a codimension two linear subspace defined over $\Q$. 
Then \eqref{G} must fail for some nonnegative integer $r\leq s$. 
\end{lemma} 

The verification of this lemma is straightforward.  After a change of coordinates,  
one may assume that the codimention two linear space
is given by $x_1=x_2=0 $.  Thus, by Hilbert's Nullstellensatz,  $C({\bf x})=0$ can now 
be written as 
\begin{equation}\label{codim2}
x_1Q_1(x_1, \dots, x_s)=x_2Q_2(x_1, \dots , x_s),  
\end{equation} 
where $Q_1$ and $Q_2$ are quadratic forms. 
From this, the lemma follows by direct calculations.  
See \cite{Zh} for details and further discussions.  

\medskip 

Thus, to conduct an investigation of asymptotic formula of rational points on 
cubic hypersurfaces in the second alternative, it seems natural to start with those of the form  \eqref{codim2}. 
A hypersurface \eqref{codim2} splits into a two-parameter family of affine 
quadrics, while the arithmetic of quadrics are well studied for centuries. It is  
therefore possible to apply, among other things, the theory of quadratic forms 
to study the density of rational points on \eqref{codim2}. 

In this paper, we pursue such an investigation when   
$C({\bf x})=0$ can further be written as 
\begin{equation}\label{Spec/Form}
x^3=Q(y_1, \ldots, y_n)z 
\end{equation} 
where $Q$ is a definite quadratic form. To simplify the details,   
we assume in particular that $Q$ takes the diagonal form $y^2 _1+\dots +y^2_n$. 
We remark that the underlying idea in treating this special case works well, 
at least in principle,  for the equation \eqref{Spec/Form}. And with more 
efforts, the same idea can be applied to establish asymptotic formulae for density 
of rational points on 
some higher-degree hypersurfaces like 
\begin{equation}\label{Sp/Form/4}
x^d=Q(y_1, \ldots, y_n)z^{d-2} 
\end{equation} 
where as before $Q$ is a positive definite quadratic form. See the forthcoming work \cite{LWZ}.

\subsection{Each $S_n$ is in the  second alternative whenever $n\geq3$}   
In this subsection, we check that, for each $s\geq 5$, each cubic hypersurface of the form 
\begin{equation} \label{SSform}
C({\bf x}):=Q(x_1, \ldots, x_s)x_2 -x_1^3 =0
\end{equation}
belongs to the second alternative,  where $Q$ is a quadratic form. Here one observes  
that \eqref{SSform} is more general since the quadratic form $Q$ may depend even on $x_1$ and $x_2$. 
In particular,  our cubic hypersurfaces
$S_n$ in \eqref{def:Sn},  with $n\geq 3$, are of the form \eqref{SSform}, and so Theorem~\ref{thm1}  is new.  
To this end,  
 we compute that 
\begin{equation}
\mathrm{Hess}(C)=
\begin{pmatrix} 
C_{11} & C_{12} & x_2Q_{13} &  \dots   & x_2Q_{1s}\\
C_{21} & C_{22}  & C_{23}  &  \dots&  C_{2s}  \\
 x_2Q_{31} & C_{32}  & x_2Q_{33}  &  \dots & x_2Q_{3s}   \\
\vdots&\vdots&\vdots&{}& \vdots \\\noalign{\vskip 1mm}
 x_2Q_{s1}  & C_{s2}  & x_2Q_{s3}   & \dots & x_2Q_{ss} \\
\end{pmatrix}, 
\end{equation}
where $C_{ij}$ stands for $\partial ^2 C/\partial x_i \partial x_j$ as usual, 
and $Q_{ij}$ has the same meaning.  
If $x_2=0$, then $\mathrm{rank}(\mathrm{Hess}(C))\leq 3$, and therefore, for some $r\leq 3$,  
\begin{eqnarray*}
&& \# \{( x_1, \cdots, x_s) \in  {\Bbb Z}^{s}: |x_i| \leq B, \,  \mathrm{rank}({\mathrm{Hess}(C)})=r\} \\ 
&&\gg B^{s-1}    
>B^{3+\varepsilon}
\end{eqnarray*}
provided $s\geq 5 $. This verifies that for each $s \geq 5$,  cubic hypersurface of the form \eqref{SSform}  falls into the second 
alternative.  

\subsection{Manin's conjecture}\label{ManinSub}
The second perspective from which Theorem~\ref{thm1} can be 
viewed is Manin's conjecture. Manin  \cite {BM} has put forward a fundamental conjecture relating the geometry of a projective variety to the distribution of its rational points. 
The original conjecture was formulated for smooth Fano varieties, 
and the number of log-powers in an asymptotic formula for the 
density of rational points is one off the rank of the Picard group. 
This has been generalized to a large class of singular Fano varieties by 
Batyrev and Tschinkel  in \cite{BT2}. Before we state this generalized  
Manin conjecture for $S_n$, let us recall the following definitions; see e.g. \cite{BT2} and \cite{Ya}. 

A normal irreducible algebraic variety $W$ is said to have 
at worst canonical singularities if $K_W$ is a $\Q$-Cartier 
divisor and if for some 
resolution of singularities $\phi : X \to W $,   one has 
$$
K_X={\phi}^* (K_W)+D, 
$$ 
where $D$ is an effective 
$\Q$-Cartier divisor. Irreducible components of the exceptional locus of $\phi$ which are not contained in the support of $D$ are called $\textit{crepant  divisors} $ of the resolution $\phi $. 
Given  $\phi :  \widetilde{S}_n \to S_n$  a resolution of 
singularities,  we denote the number of crepant divisors over $\Q$ by $\gamma_n$. Note that $\gamma_n$ is independent of the particular resolution we choose.  Also, we let
$ r_n :=\mathrm{rank}_{\Q}  \big(\text{Pic}(S_n  )\big)  $ be the Picard rank of $S_n$.  
For our $S_n$,  Manin's conjecture predicts the 
asymptotic formula 
\begin{equation}\label{MConjecture:Sn}
N_n(B)\sim C_n B(\log B)^{r_n + \gamma_n-1}
\end{equation}
as $B\to\infty$ for the quantity $N_n(B)$ defined in 
\eqref{def:Box}, where $C_n$ is a positive constant.  See Conjecture 5.6 in \cite{Ya} for the general statement. 
In \S\ref{Resolution}, we will show that $r_n=1$  and $\gamma_n=2$ whenever 
$n\geq 3$.  Hence Theorem~\ref{thm1} proves  the 
conjecture \eqref{MConjecture:Sn}. 

\medskip 

Progresses towards Manin's conjecture have been made 
for surfaces. A number of typical cases have been verified by Browning, 
de la Bret\`eche, Derenthal, Peyre and others; see the survey \cite{Br} 
for further references.  Another class of varieties have been 
extensively studied are varieties with many symmetries, e.g. toric varieties; 
see the papers \cite{FMT, BT1,  CLT} and the book  \cite{Ja} for further information. Besides these two classes of varieties, progresses have also 
been made on higher-dimensional varieties.  One example is the result 
on Segre cubic by de la Bret\`eche  \cite{Br}.  The other two examples 
are both on cubic fourfolds, 
by Schmidt \cite{Sc} and by Blomer, Br\"udern and Salberger \cite{BBS},  
respectively. The main result of this paper gives another class of 
higher-dimensional varieties on which Manin's conjecture holds. 

\medskip

We conclude this subsection by a brief discussion of  some related results. 
Other than our $S_n$ with $n\geq 3$,  the surface 
$$
S_2 : \  x^3=(y_1 ^2 + y_2 ^2)z
$$  
enjoys an  
additional  toric structure,  which is non-split over $\Q$, and therefore Manin's 
conjecture for $S_2$ follows from the general result of Batyrev and 
Tschinkel \cite{BT1}.  The closely related split toric surface
$$S'_2 :   \hskip 3mm  x^3=yzw$$  
is well studied by a number of authors.  Again, Manin's conjecture for 
$S_2'$ is a consequence of 
Batyrev and Tschinkel \cite{BT1}.  Other 
authors  include de la Bret\`eche \cite{Breteche1998},   de la 
Bret\`eche and Swinnerton-Dyer  \cite{BSD}, 
Fouvry  \cite{Fo}, Heath-Brown and Moroz \cite{HBM}  and Salberger  \cite{Sa}.  
Of the unconditional  asymptotic formulae obtained, the strongest one is 
in \cite{Breteche1998}, which
gives the estimate
$$ 
N_U(B)=BP(\log B) + O\big(B^{7/8}\exp(-c(\log B)^{3/5}(\log\log B)^{-1/5})\big),
$$ 
where $U$ is a Zariski  open subset of $S'_2$, and $P$ is a 
polynomial of degree $6$ and $c$ is a positive constant.
In \cite{BSD}, even the second term of the counting function $N_U(B)$ is established 
under the Riemann Hypothesis as well as the assumption 
that all the zeros of the Riemann zeta-function are simple.

\subsection{Outline of the proof of Theorem~\ref{thm1}} 
We are going to use the arithmetic function $r_4(d)$ defined as the 
the number representations of a positive integer $d$ as the sum of four 
squares 
\begin{equation}\label{def:r4}
d=y_1^2+\cdots+y_4^2
\quad\text{with}\quad
(y_1, \dots, y_4)\in \Z^4. 
\end{equation}
It is well-known (cf. \cite[(3.9)]{Grosswald1985}) that 
\begin{equation}\label{def:r4*}
r_4(d) = 8r_4^*(d)
\quad\text{with}\quad
r_4^*(d) := \sum_{\substack{\ell\mid d\\ \ell\not\equiv 0 ({\rm mod}\,4)}} \ell.  
\end{equation}
In view of the above, we can write
\begin{equation}\label{decomposition:N4*B}
\begin{aligned}
N^*_4(B)
& = 2\sum_{n\le B} \sum_{\substack{d\mid n^3\\ n^3/B\le d\le B^2}} r_4(d)
= 16 \sum_{n\le B} \sum_{\substack{d\mid n^3\\ n^3/B\le d\le B^2}} r_4^*(d)
\\
& = 16 \bigg(\sum_{n\le B} \sum_{\substack{d\mid n^3\\ d\le B^2}} r_4^*(d)
- \sum_{n\le B} \sum_{\substack{d\mid n^3\\ d<n^3/B}} r_4^*(d)\bigg).
\end{aligned}
\end{equation}
Hence to prove \eqref{eq:N4*B} in Theorem~\ref{thm1}, it is sufficient to establish 
asymptotic formulae for the following two quantities
\begin{equation}\label{def:Sxy}
S(x, y)
:= \sum_{n\le x} \sum_{\substack{d\mid n^3\\ d\le y}} r_4^*(d)
\quad\text{and}\quad
T(B)
:= \sum_{n\le B} \sum_{\substack{d\mid n^3\\ d<n^3/B}} r_4^*(d).
\end{equation}
In the above definition of $S(x, y)$, we have used $x$  
instead of the commonly used letter $B$, since 
in the proof we need to take integration with respect to $x$.  

In \S\S\ref{Dirichlet}-\ref{PfThm2} we shall apply analytic methods to 
establish an asymptotic formula for $S(x, y)$ 
as shown in the following. 

\begin{theorem}\label{thm2}
Let $\varepsilon>0$ be arbitrary. We have 
\begin{equation}\label{Evaluation:Sxy}
S(x, y)  
= xy \bigg(4P(\psi) + \frac{4}{3}P'(\psi)-\frac{1}{3}P''(\psi)\bigg) 
+ O_{\varepsilon}\big(x^{\frac{3}{2}} y^{\frac{3}{4}} 
+ x^{\frac{1}{2}+\varepsilon} y^{\frac{7}{6}}\big)
\end{equation}
uniformly for $x^3\ge y\ge x\ge 10$,
where $\psi:=\log x - \tfrac{1}{3}\log y$ and $P(t)$ is a quadratic 
polynomial, defined as in \eqref{def:Pt} below.
In particular, we have
\begin{equation}\label{Cor:Sxy}
S(x, y)  
= 4 \mathscr{C}_4 xy 
\bigg(\log x-\frac{1}{3}\log y\bigg)^2 \, \bigg\{1 + O\bigg(\frac{1}{\log x}\bigg)\bigg\}
\end{equation}
uniformly for $x\ge 10$ and $x^2 (\log x)^{-8}\le y\le x^2(\log x)^{4/3}$,
where 
\begin{equation}\label{def:a2}
\mathscr{C}_4 := \frac{81}{512} \zeta(4)
\prod_{p>2} \bigg(1 + \frac{2}{p} + \frac{3}{p^2}+ \frac{2}{p^3} + \frac{1}{p^4}\bigg) \bigg(1 - \frac{1}{p}\bigg)^2
\end{equation}
is the leading coefficient of $P(t)$.
\end{theorem}

The asymptotic formula \eqref{Evaluation:Sxy} is valid 
for $x^2 (\log x)^{-8}\le y\le x^{3-\varepsilon}$, which is 
sufficient for the proof of Theorem \ref{thm1}. Of course it 
is possible to extend \eqref{Evaluation:Sxy} to a wider range of $y$,  
but we shall not get into this.    

Now we turn to the evaluation of $T(B)$.   Theorem~\ref{thm2} does not apply to $T(B)$ directly since the range of its second summation depends on the variable $n$ of the first summation. Fortunately we can show that Theorem~\ref{thm2} together with some 
delicate analysis is sufficient to establish the following result.

\begin{theorem}\label{thm3}
As $B\to\infty$, we have
\begin{equation}\label{Evaluation:TB}
T(B)
= \frac{1}{9} \mathscr{C}_4 B^3(\log B)^2
\bigg\{1+O\bigg(\frac{1}{\log B}\bigg)\bigg\}, 
\end{equation}
where $\mathscr{C}_4$ is as in \eqref{def:a2}. 
\end{theorem}

Finally  in \S\ref{PfThm3} we shall deduce \eqref{eq:N4*B}  in 
Theorem~\ref{thm1} from Theorems~\ref{thm2} and \ref{thm3}, and derive 
\eqref{eq:N4B} from \eqref{eq:N4*B}. 

\medskip

Our proofs of Theorem~\ref{thm2} and Theorem~\ref{thm3}, though perhaps
similar in spirit to the arguments of de la Breteche \cite{Breteche1998}, involve
a number of  refinements and new ideas on the analytic side. 
First, the evaluation of $S(x, y)$ is much more involved and complicated  than those in \cite{Breteche1998} since we lose symmetry here.  In \cite{Breteche1998}, all the variables are symmetric 
so that application of the complex integration method is rather standard.
Besides the absence of  symmetry, 
 the inner summation range in our situation is so sparse that we cannot get any power-saving error term. 
Second,  the evaluation of $T(B)$ is quite delicate since the range of  its second summation depends on the variable $n$ of the first summation. The complex-integration method used to prove Theorem~\ref{thm2} cannot be applied to treat $T(B)$.
The main difficulty is that the error term obtained by this method is too big to get an asymptotic formula for $T(B)$.
Finally, the method we use to evaluate $T(B)$  works for other situations as well. For example, it also works
in high degree forms;  see \cite{LWZ}. We hope that this method may also be useful in evaluating similar sums from other analytic number theory questions.

\section{Resolution of singularity}\label{Resolution}
In this section we construct a resolution of $S_n$ with $n\geq 3$, and calculate its Picard rank.  Then we find out the number of  crepant divisors of this resolution.

\subsection{A  resolution of $S_n$ with $n\geq 3$} 
Let 
$$
F=(y_1^2 +y_2^2+\cdots +y_n^2)z-x^3. 
$$ 
Then 
\begin{equation*}
\begin{aligned}
\nabla F
& := \bigg(\frac{\partial F}{\partial x}, \frac{\partial F}{\partial y_1}, \dots, \frac{\partial F}{\partial y_n}, 
\frac{\partial F}{\partial z}\bigg)
\\
& \; = (-3x^2, 2y_1 z, \dots , 2y_n z, y_1^2 +\cdots +y_n^2).
\end{aligned}
\end{equation*}
Thus $S_n$ has an isolated singular point at $P := [0:0: \dots :0:1]$ and the non-isolated singular locus
$T := \{[x: y_1: \dots: y_n: z] \in S_n \,:\, x=z=y_1^2+\cdots+ y_n^2=0\}$.

We will first resolve the singularity at the isolated point $P$. 
For this purpose, we only need to resolve singularity at the affine chart $z \neq 0$. Set $z=1$. 
For the affine equation 
\begin{equation}\label{sing:point}
x^3=y_1^2 + \cdots + y_n^2, 
\end{equation}
we consider its zero locus. 
In ${\A}^{n+1}\times {\P}^n$, ${\P}^n( [u:t_1: \dots : t_n ])$, 
let 
$$
\frac{x}{u}=\frac{y_1}{t_1}=\cdots =\frac{y_n}{t_n}\cdot
$$ 
In the affine piece $u\neq 0$, set $u=1$. 
Then $y_1=t_1 x, \dots, y_n =t_n x$. 
Plugging into \eqref{sing:point}, we get an equation $$x=t_1^2+\cdots +t_n^2,$$ which is smooth. 

In the affine charts $t_i\neq 0$ for $1\le i \le n$, if we set $t_i=1$, then we have 
$$
(uy_i)^3=y_i^2(t_1^2+\cdots+t_{i-1}^2+1+t_{i+1}^2+\cdots+ t_n^2),
$$  
which gives 
$$
u^3y_i=t_1^2+\cdots+t_{i-1}^2+1+t_{i+1}^2+\cdots+ t_n^2.
$$ 
Again it defines a smooth variety. Thus after one blow-up we resolve the singularity at the 
point $P$.  
Note that the exceptional divisor $E_1$ over the singular point $P $ is the 
zero locus of the equation 
$$
t_1^2+ \cdots + t_n^2 =0 
$$
in $\P^n$, which is irreducible. 

Secondly, let us resolve the non-isolated singularities of $S_n$. 
Noticing the symmetry in the coordinates $y_1, \dots, y_n$, 
we only need to consider the affine piece $y_n\neq 0$. 
Set 
$y_n=1$. Then the equation 
\begin{equation}\label{sing:continuous}
x^3=z(y_1^2+\cdots +y_{n-1}^2+1)
\end{equation}
defines an $n$-dimensional affine variety in $\A^{n+1}$. In $\A^{n+1}\times \P^2$, $\P^2([u:v:w])$, we let 
\begin{equation}\label{ratio}
\frac{x}{u}=\frac{z}{v}=\frac{y_1^2+\cdots + y_{n-1}^2+1}{w}\cdot
\end{equation}
In the affine chart $u\neq 0$, we set $u=1$. Then $z=vx$ and
$y_1^2+\cdots +y_{n-1}^2 +1=wx$. Plugging into equation \eqref{sing:continuous}, we get
\begin{equation}\label{piece:u}
x=vw.
\end{equation}
This defines a smooth variety. 

In the affine chart $v\neq 0$, we set $v=1$. Then $x=uz$ and
$y_1^2+\cdots +y_{n-1}^2 +1=wz$. Plugging into equation \eqref{sing:continuous}, we get
\begin{equation}\label{piece:v}
u^3z=w
\end{equation}
which also gives a smooth variety. 

In the affine piece $w\neq 0$, let $w=1$. Then we get $x=u(y_1^2+\cdots +y_{n-1}^2+1)$ and
$z=v(y_1^2+\cdots +y_{n-1}^2+1)$. Inserting into \eqref{sing:continuous}, we get
\begin{equation}\label{piece:w}
u^3(y_1^2+\cdots +y_{n-1}^2+1)=v, 
\end{equation}
which again is  smooth.  
Hence the above blow-up completely resolve the non-isolated singularities.

For any given point in the non-isolated singular locus subvariety $T$, 
the inverse image of this point in the affine chart $u \neq 0$ is defined by the equation \eqref{piece:u}. We deduce that $vw=0$. Similarly, in the affine piece $v \neq 0$ and the affine piece $w \neq 0$ we get $w=0$ and $v=0$, respectively. This implies the inverse image of a point in the non-isolated singular locus is a pair of $\P^1$ and the exceptional divisor has two irreducible components $F_v$ and $F_w$, which are defined as 
$$
F_v:= \{[x: y_1: \dots: y_n : z; u: v: w] \,:\, x=z=y_1^2+\cdots+ y_n^2=v=0\}
$$
and 
$$
F_w:= \{[x: y_1: \dots: y_n : z; u: v: w] \,:\, x=z=y_1^2+\cdots+ y_n^2=w=0\},
$$
respectively.

\subsection{Calculation of the Picard rank of  $S_n$ and the number of crepant divisors.} 
In this subsection, we calculate the Picard rank of $S_n$ and the number of crepant divisors.  

We will show that $\mathrm{rank}_{\C}(\text{Pic}(S_n))\le 1$,  which in combination with the obvious lower bound   
$\mathrm{rank}_{\Q}(\text{Pic}(S_n))\geq 1$   will establish   
$r_n=\mathrm{rank}_{\Q}(\text{Pic}(S_n))= 1.$ 

We start by claculating the rank of the class group of $S_n$.   Let $H$ be the hyperplane section of $S_n$ with $z=0$, and $U$  
the complement.  Then $U$ is an affine variety defined by the equation 
$$x^3=y_1^2+\cdots + y_n^2. $$ 
By \cite[Proposition 3.1]{SS},  
$$
\mathrm{Cl}(U)\simeq 0. 
$$ 
Noticing that the hperplane section $H=\{z=x=0\} $ is irreducible and applying 
\cite[Proposition II.6.5]{Ha}, we have the exact sequence 
$$ 
\Z \to \mathrm{Cl}(S_n) \to \mathrm{Cl}(U) \to 0, 
$$ 
which gives   
$$\mathrm{rank}_{\C}(\mathrm{Cl}(S_n))\leq 1. $$
Hence, we have $$\mathrm{rank}_{\C}(\text{Pic}(S_n)) \le \mathrm{rank}_{\C}(\mathrm{Cl}(S_n))\le 1. $$

In the following, we calculate the number of linearly independent crepant divisors for the resolution in the previous sub-section. 
 
 Let $\phi: \widetilde{S}_n \to S_n$ be the resolution map, where $\widetilde{S}_n$  is the  desingularisation of $S_n$.  Define
\begin{equation*}
\begin{aligned}
E_1 
& := \phi^{-1}(P),
\\
E_2
& := \phi^{-1}(T)=F_v + F_w,
\\ 
L 
& := \phi^*(\OO (1)|_{S_n}).
\end{aligned}
\end{equation*}

Let $\pi: \widetilde{\P}^{n+1}\to \P^{n+1}$ be the projection map, 
where $\widetilde{\P}^{n+1}=\Bl_{\{P, T\}}(\P^{n+1})$ is the  blow-up of $\P^{n+1}$. 
Then $\phi: \widetilde{S}_n \to S_n$ is the restriction of $\pi$ to $\widetilde{S}_n$.   We have the following natural commutative diagram: 
$$
\begin{tikzcd}
\widetilde{S}_n \arrow [hookrightarrow]{r}  \arrow[swap]{d}{\phi} & \widetilde{\P}^{n+1} \arrow{d}{\pi} \\
S_n \arrow [hookrightarrow] {r} & \P^{n+1}
\end{tikzcd}
$$ 
Let  $\widetilde{E}_1 := \pi^{-1}(P)$ and $\widetilde{E}_2 := \pi^{-1}(T)$.  
Then we have 
$$ 
K_{\widetilde{\P}^{n+1}}=\pi^* K_{\P^{n+1}} + n\widetilde{E}_1 + 2 \widetilde{E}_2 
$$ 
and 
$$
\pi^*S_n=\widetilde{S}_n +2\widetilde{E}_1 + 2\widetilde{E}_2, 
$$ 
where the last equality follows from the fact that $S_n$ has multiplicity two both at $P$ and the subvariety $T$.  By the adjunction formula, we get
 $$K_{\widetilde{S}^{n}}=(K_{\widetilde{\P}^{n+1}}+\widetilde{S}_n)|_{\widetilde{S}_n}
 = \big(\pi^* (K_{\P^{n+1}+S_n}) + (n-2)\widetilde{E}_1\big)|_{\widetilde{S}_n}=\pi^* K_{S_n} + (n-2)E_1, $$
 since $\widetilde{E}_1 \cap \widetilde{S}_n=E_1$.
Recall our assumption $n\geq 3$.  Then  we notice that $\widetilde{E}_2$ does not appear in the canonical divisor of $\widetilde{S}_n$.  Also,  $\widetilde{E}_2\cap \widetilde{S}_n=F_u \cup F_v$. 
Therefore, we conclude that the linearly independent exceptional divisors $F_u$ and 
$F_v$ are crepant, and conclude the following. 

\begin{proposition}\label{F:CreRank}
For all $n\geq 3$, we have
$$
r_n =\mathrm{rank}_{\Q} \big(\mathrm{Pic}(S_n )\big)=1  \hskip 5mm \mathrm{and}  \hskip 5mm \gamma_n=2. 
$$
\end{proposition}

\section{Dirichlet series associated with $S(x, y)$}\label{Dirichlet}

In view of the definition of $S(x, y)$ in \eqref{def:Sxy}, we define 
the double Dirichlet series 
\begin{equation}\label{def:Fsw}
\mathcal{F}(s, w) := \sum_{n\ge 1} n^{-s} \sum_{d\mid n^3} d^{-w} r_4^*(d)
\end{equation}
for $\re s>4$ and $\re w>0$, where $s$ and $w$ are complex parameters. 
The next lemma states that the function 
$\mathcal{F}(s, w) $ enjoys a nice factorization formula. 
In the following and throughout the paper, we denote by $\zeta(s)$ the Riemann zeta-function
and by $\tau(n)$ the divisor function.  

\begin{lemma}\label{Lem:Fsw}
For $\min_{0\le j\le 3} \re (s+jw-j)>1$, 
we have
\begin{equation}\label{Expression:Fsw}
\mathcal{F}(s, w) = \prod_{0\le j\le 3} \zeta(s+jw-j) \mathcal{G}(s, w),
\end{equation}
where $\mathcal{G}(s, w)$ is an Euler product, given by \eqref{def:Gpsw}, \eqref{def:G2sw} and \eqref{def:Gsw} below.
Further, for any $\varepsilon>0$, 
$\mathcal{G}(s, w)$ converges absolutely
for $\min_{0\le j\le 3} \re (s+jw-j)\ge \tfrac{1}{2}+\varepsilon$, and
in this half-plane
\begin{equation}\label{UB:Gsw}
\mathcal{G}(s, w)\ll_{\varepsilon} 1.
\end{equation}
\end{lemma}

\begin{proof}
Obviously the 
functions $r_4^*(d)$ and  $n^{-s} \sum_{d\mid n^3} d^{-w} r_4^*(d)$ 
are multiplicative.
Since $r_4^*(d)\le d\tau(d)$, 
for $\re s>4$ and $\re w>0$ we can write the Euler product
$$
\mathcal{F}(s, w)
= \prod_p \sum_{\nu\ge 0} p^{-\nu s} \sum_{0\le \mu\le 3\nu} p^{- \mu w} r_4^*(p^{\mu})
=: \prod_p \mathcal{F}_p(s, w).
$$
The next is to simplify each $\mathcal{F}_p(s, w)$. To this end, we recall 
\eqref{def:r4*} so that 
\begin{equation}\label{def:r4*pmu}
\begin{aligned}
r_4^*(p^{\mu})
& \;= \begin{cases}
\dfrac{1-p^{\mu+1}}{1-p}  & \text{if } p>2, 
\\\noalign{\vskip 0mm}
3                                      & \text{if } p=2,  
\end{cases}
\end{aligned}
\end{equation}
for all integers $\mu\ge 1$. On the other hand, a simple formal calculation shows
\begin{equation}\label{Formal_Calcul_1}
\begin{aligned}
& \sum_{\nu\ge 0} x^{\nu} \sum_{0\le \mu\le 3\nu} y^{\mu} \frac{1-z^{\mu+1}}{1-z} 
\\
& = \frac{1}{1-z} \sum_{\nu\ge 0} x^{\nu} 
\bigg(\frac{1-y^{3\nu+1}}{1-y} - z\frac{1-(yz)^{3\nu+1}}{1-yz}\bigg)
\\
& = \frac{1}{1-z} \bigg\{
\frac{1}{1-y}\bigg(\frac{1}{1-x} - \frac{y}{1-xy^3}\bigg)
- \frac{z}{1-yz}\bigg(\frac{1}{1-x} - \frac{yz}{1-xy^3z^3}\bigg)
\bigg\}
\\\noalign{\vskip 1mm}
& = \frac{1+xy(1+z)+xy^2(1+z+z^2)+xy^3(z+z^2)+x^2y^4z^2}{(1-x)(1-xy^3)(1-xy^3z^3)}, 
\end{aligned}
\end{equation}
and
\begin{equation}\label{Formal_Calcul_2}
\begin{aligned}
1 + \sum_{\nu\ge 1} x^{\nu} 
\Big(1 + a \sum_{1\le \mu\le 3\nu} y^{\mu}\Big)
& = 1 + \sum_{\nu\ge 1} x^{\nu} 
\bigg(1 + a \frac{y-y^{3\nu+1}}{1-y}\bigg)
\\\noalign{\vskip -0,6mm}
& = \frac{1}{1-x}
+ \frac{a}{1-y} \bigg(\frac{xy}{1-x} - \frac{xy^{4}}{1-xy^3}\bigg)
\\\noalign{\vskip 0,5mm}
& = \frac{1+axy(1+y)+(a-1)xy^3}{(1-x)(1-xy^3)}\cdot
\end{aligned}
\end{equation}
When $p>2$, in view of \eqref{def:r4*pmu},
we can apply \eqref{Formal_Calcul_1} with $(x, y, z) = (p^{-s}, p^{-w}, p)$ to write
\begin{equation}\label{Fpsw:p>2}
\mathcal{F}_p(s, w)
= \prod_{0\le j\le 3} \big(1-p^{-(s+jw-j)}\big)^{-1} \mathcal{G}_p(s, w),
\end{equation}
where
\begin{equation}\label{def:Gpsw}
\begin{aligned}
\mathcal{G}_p(s, w)
& := \bigg(1
+\frac{p+1}{p^{s+w}}
+\frac{p^2+p+1}{p^{s+2w}}
+\frac{p^2+p}{p^{s+3w}}
+\frac{p^2}{p^{2s+4w}}\bigg)
\\
& \qquad\times
\bigg(1-\frac{p}{p^{s+w}}\bigg)
\bigg(1-\frac{p^2}{p^{s+2w}}\bigg)
\bigg(1-\frac{1}{p^{s+3w}}\bigg)^{-1}
\cdot
\end{aligned}
\end{equation}
While for $p=2$, the formula \eqref{Formal_Calcul_2} 
with $(x, y, z, a) = (2^{-s}, 2^{-w}, 2, 3)$ gives us
\begin{equation}\label{Fpsw:p=2}
\mathcal{F}_2(s, w)
= \prod_{0\le j\le 3} \big(1-2^{-(s+jw-j)}\big)^{-1} \mathcal{G}_2(s, w),
\end{equation}
where
\begin{equation}\label{def:G2sw}
\mathcal{G}_2(s, w)
:= \frac{1+3\cdot 2^{-s-w}+3\cdot 2^{-s-2w}+2^{-s-3w+1}}{1-2^{-s-3w}} \prod_{1\le j\le 3} (1-2^{-(s+jw-j)}).
\end{equation}
Combining \eqref{Fpsw:p>2}--\eqref{def:G2sw}, we get \eqref{Expression:Fsw} with
\begin{equation}\label{def:Gsw}
\mathcal{G}(s, w) := \prod_p \mathcal{G}_p(s, w)
\end{equation}
for $\re s>4$ and $\re w>0$.

Next we prove \eqref{UB:Gsw}.
It is easy to verify that
for $\displaystyle\min_{0\le j\le 3} (\sigma+ju-j)\ge \tfrac{1}{2}+\varepsilon$, we have
\begin{align*}
2(\sigma+u-1)
& \ge 2(\tfrac{1}{2}+\varepsilon) = 1+2\varepsilon,
\\
2(\sigma+2u-2)
& \ge 2(\tfrac{1}{2}+\varepsilon) = 1+2\varepsilon,
\\
\sigma+u
& \ge 1+\tfrac{1}{2}+\varepsilon = \tfrac{3}{2}+\varepsilon,
\\
\sigma+2u
& \ge 2+\tfrac{1}{2}+\varepsilon = \tfrac{5}{2}+\varepsilon,
\\
\sigma+2u-1
& \ge 1+\tfrac{1}{2}+\varepsilon = \tfrac{3}{2}+\varepsilon,
\\
\sigma+3u-1
& \ge 2+\tfrac{1}{2}+\varepsilon = \tfrac{5}{2}+\varepsilon,
\\
\sigma+3u-2
& \ge 1+\tfrac{1}{2}+\varepsilon = \tfrac{3}{2}+\varepsilon,
\\
2(\sigma+2u-1)
& \ge 2(1+\tfrac{1}{2}+\varepsilon) = 3+\varepsilon,
\\
\sigma+3u
& \ge 3+\tfrac{1}{2}+\varepsilon = \tfrac{7}{2}+\varepsilon.
\end{align*}
These together with \eqref{def:Gpsw} imply that
$$
|\mathcal{G}_p(s, w)|
= 1 + O(p^{-1-\varepsilon})
$$ 
for $\min_{0\le j\le 3} \re(s+jw-j)\ge \tfrac{1}{2}+\varepsilon$.
This shows that the Euler product $\mathcal{G}(s, w)$ converges absolutely
for $\min_{0\le j\le 3} \re (s+jw-j)\ge \tfrac{1}{2}+\varepsilon$, 
and \eqref{UB:Gsw} holds in this half-plane.
By analytic continuation, \eqref{Expression:Fsw} is also true in the same domain.
This completes the proof.
\end{proof}

\section{Outline of the proof of Theorem~\ref{thm2}} 
The basic idea is to apply the method of complex integration to, 
instead of our original $S(x, y)$,  the quantity 
\begin{equation}\label{def:MXY}
M(X, Y) := \int_1^Y \int_1^X S(x, y) \d x \d y 
\end{equation}
which is a mean-value of $S(x, y)$. This $M(X, Y)$ 
is much easier to handle; in particular when moving the contours  
of integration to the left, this does not involve any problem of convergence.  
We will first establish an asymptotic formula for $M(X, Y)$, and then 
derive the asymptotic formula \eqref{Evaluation:Sxy} for $S(x, y)$ 
in Theorem~\ref{thm2} by an analytic argument involving 
the operator $\mathscr{D}$ defined in the next paragraph. 
If each of these sums $S$ and $M$ has just one variable, 
the above method has been known for a long time;  
we refer the readers to \cite[Chapter II.5]{Tenenbaum1995} for an 
excellent exposition.  De la Bret\`{e}che \cite{Breteche1998} successfully 
handled a case where each of these sums $S$ and $M$ has three variables. 
In our present situation each of these sums $S$ and $M$ has two variables. 

Denote by $\mathscr{E}_k$ the set of all functions of $k$ variables
and define the operator $\mathscr{D}: \mathscr{E}_2\to \mathscr{E}_4$ 
by
\begin{equation}\label{def:Tf}
(\mathscr{D}f)(X, H; Y, J) := f(H, J) - f(H, Y) - f(X, J) + f(X, Y).
\end{equation} 
Our $S(x, y)$ and $M(X, Y)$ are closed related as shown in the following lemma, 
which in particular enables one to derive an asymptotic formula for $S(x, y)$ 
from that for $M(x, y)$. 

\begin{lemma}\label{Operator/1} 
Let $S(x, y)$ and $M(X, Y)$ be defined as in \eqref{def:Sxy} and \eqref{def:MXY}.
Then 
$$
(\mathscr{D}M)(X-H, X; Y-J, Y)\le HJS(X, Y)\le (\mathscr{D}M)(X, X+H; Y, Y+J)
$$
for $H\le X$ and $J\le Y$.  
\end{lemma}

The operator $\mathscr{D}$ has some properties that we are going to use repeatedly 
throughout the paper. These are summarized in the following lemma.   

\begin{lemma}\label{Operator/2}
\par
{\rm (i)}
Let $f\in \mathscr{E}_2$ be a function of class $C^3$. Then we have
$$
(\mathscr{D}f)(X, H; Y, J)
= (J-Y)(H-X) \bigg\{\frac{\partial^2f}{\partial x\partial y}(X, Y) + O\big(R(X, H; Y, J)\big)\bigg\}
$$
for $X\le H$ and $Y\le J$, where
$$
R(X, H; Y, J) 
:= (H-X) 
\max_{\substack{X\le x\le H\\ Y\le y\le J}} \bigg|\frac{\partial^3f}{\partial x^2\partial y}(x, y)\bigg|
+ (J-Y)\max_{\substack{X\le x\le H\\ Y\le y\le J}} \bigg|\frac{\partial^3f}{\partial x\partial y^2}(x, y)\bigg|.
$$
\par
{\rm (ii)}
If $f(X, Y) = f_1(X)f_2(Y)$, then
$$
(\mathscr{D}f)(X, H; Y, J) = \big(f_1(H)-f_1(X)\big) \big(f_2(J)-f_2(Y)\big).
$$
\end{lemma}

Lemmas~\ref{Operator/1} and \ref{Operator/2} can be proved similarly as in  
\cite[Lemma 2]{Breteche1998}; the details are therefore omitted. 

The next elementary estimate will also be used several times in the paper. 
It is essentially \cite[Lemma 6(i)]{Breteche1998}.  

\begin{lemma}\label{Lem3.2}
Let $1\le H\le X$ and $|\sigma|\le 10$.
Then for any $\beta\in [0, 1]$, we have
\begin{equation}\label{Lem3.2_Eq_A}
\big|(X+H)^{s} - X^{s}\big|\ll X^{\sigma} ((|\tau|+1)H/X)^{\beta},
\end{equation}
where the implied constant is absolute.
\end{lemma}

\begin{proof}
We have trivially $\big|(X+H)^{s} - X^{s}\big|\ll X^{\sigma}$.
On the other hand, we can write
$$
\big|(X+H)^{s} - X^{s}\big|
= \bigg|s \int_{X}^{X+H} x^{s-1} \d x\bigg|
\ll |s| X^{\sigma-1} H.
$$
From these we can deduce, for any $\beta\in [0, 1]$,
$$
\big|(X+H)^{s} - X^{s}\big|
\ll (X^{\sigma})^{1-\beta} (|s| X^{\sigma-1} H)^{\beta}.
$$
This implies the desired inequality. 
\end{proof}

\section{Proof of Theorem \ref{thm2}}\label{PfThm2}

We shall first evaluate $M(X, Y)$, from which we shall deduce 
Theorem \ref{thm2} by applying the operator $\mathscr{D}$ 
defined as in \eqref{def:Tf}. In the sequel, we suppose 
\begin{equation}\label{Condition:XYTUHJ}
10\le X\le Y\le X^3,
\quad
(XY)^3\le 4T\le U\le X^{12},
\quad
H\le X,
\quad
J\le Y,
\end{equation}
and for brevity we fix the following notation:
\begin{equation}\label{def:kappa_Lambda_L}
s := \sigma+\mathrm{i}\tau,
\quad
w := u+\mathrm{i}v,
\quad
\mathcal{L} := \log X,
\quad
\kappa := 1+\mathcal{L}^{-1},
\quad
\lambda := 1+4\mathcal{L}^{-1}.
\end{equation}
The following proposition is an immediate consequence of 
Lemmas \ref{Perron_Formula:MXY}-\ref{Lem:Evaluate_I3} below. 

\begin{proposition}\label{Pro:M1XY}
Under the previous notation, we have
$$
M(X, Y)
= X^2 Y^2 P \bigg(\log X-\frac{1}{3}\log Y\bigg) + R_0(X, Y) + \cdots + R_3(X, Y) + O(1)
$$
uniformly for $(X, Y, T, U, H, J)$ satisfying 
\eqref{Condition:XYTUHJ},
where $R_0, R_1, R_2, R_3$ and $P(t)$ are defined as in 
\eqref{def:R0XY}, 
\eqref{def:R1XY}, 
\eqref{def:R2XY}, 
\eqref{def:R3XY}
and \eqref{def:Pt} below,
respectively.
\end{proposition}

The proof is divided into several subsections. 

\subsection{Application of Perron's formula} 
The first step is to apply Perron's formula twice to transform 
$M(X, Y)$ into a form that is ready for future treatment.  

\begin{lemma}\label{Perron_Formula:MXY}
Under the previous notation, we have
\begin{equation}\label{MXY=MXYTU}
M(X, Y)
= M(X, Y; T, U) + O(1)
\end{equation}
uniformly for $(X, Y, T, U)$ satisfying 
\eqref{Condition:XYTUHJ}, where the implied constant is absolute and
\begin{equation}\label{def:M1XYTU}
M(X, Y; T, U)
:= \frac{1}{(2\pi {\rm i})^2} \int_{\kappa-{\rm i}T}^{\kappa+{\rm i}T} 
\bigg(\int_{\lambda-{\rm i}U}^{\lambda+{\rm i}U} \frac{\mathcal{F}(s, w) Y^{w+1}}{w(w+1)} \d w\bigg) 
\frac{X^{s+1}}{s(s+1)} \d s.
\end{equation}
\end{lemma}

\begin{proof}
In view of the definition of $r_4^*(d)$, we have, for any $\varepsilon>0$ and all $d\geq 1$, 
$$
r_4^*(d)\le d\tau(d)\ll_{\varepsilon} d^{1+\varepsilon}, 
$$
which implies that 
$$
\tau_*(n^3, y) 
:= \sum_{d\mid n^3, \, d\le y} r_4^*(d)
\ll_{\varepsilon} y^{1+\varepsilon}\tau(n^3)
\ll_{\varepsilon} y^{1+\varepsilon} n^{\varepsilon}
$$
uniformly for $y\ge 1$ and $n\in \N$, where the implied constant depends on $\varepsilon$ only.
Thus the Dirichlet series $\sum_{n\ge 1} \tau_*(n^3, y) n^{-s}$ converges absolutely for $\sigma>1$.
Applying Perron's formula \cite[Theorem II.2.3]{Tenenbaum1995}, we write 
\begin{equation}\label{Perron_s}
\int_1^X S(x, y) \d x
= \frac{1}{2\pi {\rm i}} \int_{\kappa-{\rm i}\infty}^{\kappa+{\rm i}\infty}
\sum_{n\ge 1} \frac{\tau_*(n^3, y)}{n^s} \frac{X^{s+1}}{s(s+1)} \d s, 
\end{equation}
which holds for all $y\ge 1$. 

We are going to apply Perron's formula again but to the function $\tau_*(n^3; y)$ 
in the above formula. We write 
$$
\tau_*(n^3; y) = \sum_{d\le y} a_n(d) 
\qquad \mbox{with} \qquad 
a_n(d) := 
\begin{cases}
r_4^*(d)  & \text{if  } d\mid n^3, 
\\
0             & \text{otherwise}, 
\end{cases}
$$
and notice that the (finite) Dirichlet series $\sum_{d\ge 1} a_n(d) d^{-w}$ 
converges absolutely for all $w\in \C$.
Thus we have, as before,
\begin{equation}\label{Perron_w}
\int_1^Y \tau_*(n^3; y) \d y
= \frac{1}{2\pi {\rm i}} \int_{\lambda-{\rm i}\infty}^{\lambda+{\rm i}\infty}
\sum_{d\ge 1} \frac{a_n(d)}{d^w}  \frac{Y^{w+1}}{w(w+1)} \d w.
\end{equation}
Integrating \eqref{Perron_s} with respect to $y$ on $[1, Y]$ 
and then applying \eqref{Perron_w}, we find that
\begin{equation}\label{Formula_MXY}
M(X, Y)
= \frac{1}{(2\pi {\rm i})^2} 
\int_{\kappa-{\rm i}\infty}^{\kappa+{\rm i}\infty} 
\bigg(\int_{\lambda-{\rm i}\infty}^{\lambda+{\rm i}\infty} \frac{\mathcal{F}(s, w) Y^{w+1}}{w(w+1)} \d w\bigg)
\frac{X^{s+1}}{s(s+1)} \d s.
\end{equation}

Next we shall cut the above to infinite integrals into finite ones. 
We need the well-known estimate 
(cf. e.g. \cite[page 146, Theorem II.3.7]{Tenenbaum1995}) 
\begin{equation}\label{UB:zeta}
\zeta(s)\ll |\tau|^{\max\{(1-\sigma)/3, 0\}} \log |\tau|
\qquad
(\sigma\ge 1-c/\log|\tau|, \; |\tau|\ge 2) 
\end{equation}
where $c>0$ is a positive constant, as well as the fact that 
$s=1$ is the simple pole of $\zeta(s)$. From these and 
\eqref{UB:Gsw} of Lemma \ref{Lem:Fsw}, we deduce that, 
uniformly for $\tau\in \R$ and $v\in \R$, 
$$
\mathcal{F}(\kappa+\text{i}\tau, \lambda+\text{i}v)
\ll \max\{\mathcal{L}^4, \, \log^4(|\tau|+|v|+3)\}. 
$$
It follows that 
\begin{align*}
\int_{\kappa-{\rm i}\infty}^{\kappa+{\rm i}\infty} 
\bigg(\int_{\lambda\pm{\rm i}U}^{\lambda\pm{\rm i}\infty} \frac{\mathcal{F}(s, w) Y^{w+1}}{w(w+1)} \d w\bigg)
\frac{X^{s+1}}{s(s+1)} \d s
& \ll \frac{X^2Y^2\mathcal{L}^4}{U}
\ll 1,
\\
\int_{\kappa\pm{\rm i}T}^{\kappa\pm{\rm i}\infty} 
\bigg(\int_{\lambda-{\rm i}U}^{\lambda+{\rm i}U} \frac{\mathcal{F}(s, w) Y^{w+1}}{w(w+1)} \d w\bigg)
\frac{X^{s+1}}{s(s+1)} \d s
& \ll \frac{X^2Y^2\mathcal{L}^4}{T}
\ll 1.
\end{align*}
The desired formula \eqref{MXY=MXYTU} follows 
from \eqref{Formula_MXY} and the two estimates above. 
\end{proof}

\subsection{Application of Cauchy's theorem} 
In this subsection, we shall apply Cauchy's theorem to 
evaluate the integral over $w$ in $M(X, Y; T, U)$.
We write
\begin{equation}\label{def:wjs} 
w_j = w_j(s) := (j+1-s)/j
\quad
(1\le j\le 3) 
\end{equation} 
and
\begin{equation}\label{def:Fk*} 
\begin{cases}
\mathcal{F}_1^*(s)
 := \zeta(s) \zeta(2-s) \zeta(3-2s) \mathcal{G}(s, w_1(s)),
\\\noalign{\vskip 0,8mm}
\mathcal{F}_2^*(s)
 := \zeta(s) \zeta(\tfrac{s+1}{2}) \zeta(\tfrac{3-s}{2}) \mathcal{G}(s, w_2(s)),
\\\noalign{\vskip 0,8mm}
\mathcal{F}_3^*(s)
 := \zeta(s) \zeta(\tfrac{2s+1}{3}) \zeta(\tfrac{s+2}{3}) \mathcal{G}(s, w_3(s)).
\end{cases}
\end{equation}

\begin{lemma}\label{Lem:MXYTU}
Under the previous notation, for any $\varepsilon>0$ we have
\begin{equation}\label{Evaluate:M1XYTU}
M(X, Y; T, U)
= I_1 + I_2 + I_3 + R_0(X, Y) + O_{\varepsilon}(1)
\end{equation}
uniformly for $(X, Y, T, U)$ satisfying \eqref{Condition:XYTUHJ}, 
where 
\begin{align*}
I_1
& :=  \frac{1}{2\pi {\rm i}} 
\int_{\kappa-{\rm i}T}^{\kappa+{\rm i}T} \frac{\mathcal{F}_1^*(s) X^{s+1}Y^{3-s}}{(2-s)(3-s)s(s+1)} \d s, 
\\\noalign{\vskip 1mm}
I_2
& :=  \frac{4}{2\pi {\rm i}} \int_{\kappa-{\rm i}T}^{\kappa+{\rm i}T} 
\frac{\mathcal{F}_2^*(s) X^{s+1} Y^{(5-s)/2}}{(3-s)(5-s)s(s+1)} \d s, 
\\\noalign{\vskip 1mm}
I_3
& :=  \frac{9}{2\pi {\rm i}} \int_{\kappa-{\rm i}T}^{\kappa+{\rm i}T} 
\frac{\mathcal{F}_3^*(s) X^{s+1} Y^{(7-s)/3}}{(4-s)(7-s)s(s+1)} \d s, 
\end{align*}
and
\begin{equation}\label{def:R0XY}
R_{ 0}(X, Y)
:= \frac{1}{(2\pi\mathrm{i})^2} \int_{\kappa-{\rm i}T}^{\kappa+{\rm i}T} 
\bigg(\int_{\frac{11}{12}+\varepsilon-{\rm i}U}^{\frac{11}{12}+\varepsilon+{\rm i}U} \frac{\mathcal{F}(s, w) Y^{w+1}}{w(w+1)} \d w\bigg)
\frac{X^{s+1}}{s(s+1)} \d s.
\end{equation}
Further we have
\begin{equation}\label{UB_TR0}
\left.
\begin{array}{rl}
(\mathscr{D}R_{ 0})(X, X+H; Y, Y+J)\!
\\\noalign{\vskip 1mm}
(\mathscr{D}R_{ 0})(X-H, X; Y-J, Y)\!
\end{array}
\right\}
\ll_{\varepsilon} X^{\frac{7}{6}+\varepsilon} Y^{\frac{11}{12}+\varepsilon} H^{\frac{5}{6}} J 
+ X^{1+\varepsilon} Y^{\frac{13}{12}+\varepsilon} H J^{\frac{5}{6}}
\end{equation}
uniformly for $(X, Y, T, U, H, J)$ satisfying 
\eqref{Condition:XYTUHJ}.
Here the implied constants depend on $\varepsilon$ only.
\end{lemma}

\begin{proof}
We want to calculate the integral
$$
\frac{1}{2\pi {\rm i}} \int_{\lambda-{\rm i}U}^{\lambda+{\rm i}U}
\frac{\mathcal{F}(s, w) Y^{w+1}}{w(w+1)} \d w
$$
for any individual $s=\sigma+\mathrm{i}\tau$ with $\sigma=\kappa$ and $|\tau|\le T$. 
We move the line of integration $\re w = \lambda$ to $\re w=\tfrac{3}{4}+\varepsilon$.
By Lemma \ref{Lem:Fsw},  for $\sigma=\kappa$ and $|\tau|\le T$,
the points $w_j(s) \; (j=1, 2, 3)$, given by \eqref{def:wjs},
are the simple poles of the integrand in the rectangle 
$\tfrac{3}{4}+\varepsilon\le u\le \lambda$ and $|v|\le U$.
The residues of $\frac{\mathcal{F}(s, w)}{w(w+1)} Y^{w+1}$ at the poles $w_j(s)$ are
\begin{equation}\label{def:residue}
\frac{\mathcal{F}_1^*(s)Y^{3-s}}{(2-s)(3-s)}, 
\qquad
\frac{4\mathcal{F}_2^*(s)Y^{(5-s)/2}}{(3-s)(5-s)},
\qquad
\frac{9\mathcal{F}_3^*(s)Y^{(7-s)/3}}{(4-s)(7-s)},
\end{equation}
respectively, where $\mathcal{F}_j^*(s) (j=1, 2, 3)$ are defined as in \eqref{def:Fk*}.

When $\sigma=\kappa$ and $\tfrac{11}{12}+\varepsilon\le u\le \lambda$, 
it is easily checked that 
$$
\min (\sigma+ju-j)
\ge 1+3(\tfrac{11}{12}+\varepsilon-1)
=\tfrac{3}{4}+3\varepsilon
>\tfrac{1}{2}+\varepsilon.
$$
It follows from \eqref{UB:zeta} and \eqref{UB:Gsw} that,  
for $\sigma=\kappa, |\tau|\le T, 
\tfrac{11}{12}+\varepsilon\le u\le \lambda$ and $v=\pm U$, 
$$
\mathcal{F}(s, w)
\ll_{\varepsilon} U^{2(1-u)}\mathcal{L}^4. 
$$
This implies that
$$
\int_{\frac{11}{12}+\varepsilon\pm{\rm i}U}^{\lambda\pm{\rm i}U} \frac{\mathcal{F}(s, w) Y^{w+1}}{w(w+1)} \d w 
\ll_{\varepsilon} Y\mathcal{L}^4 \int_{\frac{11}{12}}^{\lambda}\bigg(\frac{Y}{U^2}\bigg)^u \d u
\ll_{\varepsilon} \frac{Y^{\frac{23}{12}}\mathcal{L}^4}{U^{\frac{11}{6}}}
\ll_{\varepsilon} 1.
$$
Cauchy's theorem then gives 
\begin{align*}
\frac{1}{2\pi {\rm i}} 
\int_{\lambda-{\rm i}U}^{\lambda+{\rm i}U} \frac{\mathcal{F}(s, w) Y^{w+1}}{w(w+1)} \d w
& = \frac{\mathcal{F}_1^*(s)Y^{3-s}}{(2-s)(3-s)} 
+ \frac{4\mathcal{F}_2^*(s)Y^{(5-s)/2}}{(3-s)(5-s)}
+ \frac{9\mathcal{F}_3^*(s)Y^{(7-s)/3}}{(4-s)(7-s)}
\\
& \quad
+ \frac{1}{2\pi {\rm i}} \int_{\frac{11}{12}+\varepsilon-{\rm i}U}^{\frac{11}{12}+\varepsilon+{\rm i}U} 
\frac{\mathcal{F}(s, w) Y^{w+1}}{w(w+1)} \d w
+ O_{\varepsilon}(1).
\end{align*}
Inserting the last formula 
into \eqref{def:M1XYTU}, we obtain \eqref{Evaluate:M1XYTU}.

Finally we prove \eqref{UB_TR0}.
For $\sigma=\kappa$, $|\tau|\le T$, $u=\tfrac{11}{12}+\varepsilon$ and $|v|\le U$, 
we apply \eqref{UB:zeta} and \eqref{UB:Gsw} as before, to get 
$$
\mathcal{F}(s, w)
\ll (|\tau|+|v|+1)^{\frac{1}{6}}\mathcal{L}^4
\ll \big\{(|\tau|+1)^{\frac{1}{6}} + (|v|+1)^{\frac{1}{6}}\big\}\mathcal{L}^4. 
$$
Also, for $\sigma, \tau, u, v$ as above, we have 
\begin{align*}
r_{s, w}(X, H; Y, J)
& := \big((X+H)^{s+1}-X^{s+1}\big) \big((Y+J)^{w+1}-Y^{w+1}\big)
\\
& \ll X^2((|\tau|+1)H/X))^{\frac{5}{6}-\varepsilon} Y^{\frac{23}{12}+\varepsilon} ((|v|+1)J/Y)^{1-\varepsilon}
\\
& \ll X^{\frac{7}{6}+\varepsilon} Y^{\frac{11}{12}+\varepsilon} H^{\frac{5}{6}} J (|\tau|+1)^{\frac{5}{6}-\varepsilon} (|v|+1)^{1-\varepsilon}
\end{align*}
by \eqref{Lem3.2_Eq_A} of Lemma \ref{Lem3.2} with 
$\beta=\tfrac{5}{6}-\varepsilon$ and with $\beta=1-\varepsilon$.
Similarly, 
\begin{align*}
r_{s, w}(X, H; Y, J)
& = \big((X+H)^{s+1}-X^{s+1}\big) \big((Y+J)^{w+1}-Y^{w+1}\big)
\\
& \ll X^2((|\tau|+1)H/X))^{1-\varepsilon} Y^{\frac{23}{12}+\varepsilon} ((|v|+1)J/Y)^{\frac{5}{6}-\varepsilon}
\\
& \ll X^{1+\varepsilon} Y^{\frac{13}{12}+\varepsilon} H J^{\frac{5}{6}} 
(|\tau|+1)^{1-\varepsilon} (|v|+1)^{\frac{5}{6}-\varepsilon}
\end{align*}
by \eqref{Lem3.2_Eq_A} of Lemma \ref{Lem3.2} with $\beta=1-\varepsilon$ 
and with  $\beta=\tfrac{5}{6}-\varepsilon$. 
These and Lemma \ref{Operator/2}(i) imply 
\begin{align*}
(\mathscr{D}R_0)(X, X+H; Y, Y+J)
& = \frac{1}{(2\pi\mathrm{i})^2}
\int_{\kappa-{\rm i}T}^{\kappa+{\rm i}T} \int_{\frac{11}{12}+\varepsilon-{\rm i}U}^{\frac{11}{12}+\varepsilon+{\rm i}U} 
\mathcal{F}(s, w) 
\frac{r_{s, w}(X, H; Y, J)}{s(s+1)w(w+1)} \d w \d s
\\\noalign{\vskip 1mm}
& \ll_{\varepsilon} X^{\frac{7}{6}+\varepsilon} Y^{\frac{11}{12}+\varepsilon} H^{\frac{5}{6}} J 
+ X^{1+\varepsilon} Y^{\frac{13}{12}+\varepsilon} H J^{\frac{5}{6}}.
\end{align*}
This completes the proof.
\end{proof}

\subsection{Evaluation of $I_1$}\

\vskip 1mm

\begin{lemma}\label{Lem:Evaluate_I1}
Under the previous notation, for any $\varepsilon>0$ we have
\begin{equation}\label{Evaluate:I1}
I_1 = R_1(X, Y) + O_{\varepsilon}(1)
\end{equation}
uniformly for $(X, Y, T)$ satisfying \eqref{Condition:XYTUHJ},
where 
\begin{equation}\label{def:R1XY}
R_1(X, Y)
:= \frac{1}{2\pi {\rm i}} \int_{\frac{5}{4}-\varepsilon-{\rm i}T}^{\frac{5}{4}-\varepsilon+{\rm i}T} 
\frac{\mathcal{F}_1^*(s)X^{s+1}Y^{3-s}}{(2-s)(3-s)s(s+1)} \d s.
\end{equation}
Further we have
\begin{equation}\label{UB_TR1}
\left.
\begin{array}{rl}
(\mathscr{D}R_1)(X, X+H; Y, Y+J)\!
\\\noalign{\vskip 1mm}
(\mathscr{D}R_1)(X-H, X; Y-J, Y)\!
\end{array}\right\}
\ll_{\varepsilon} X^{\frac{5}{4}} Y^{\frac{3}{4}+\varepsilon} H J
\end{equation}
uniformly for $(X, Y, T, H, J)$ satisfying \eqref{Condition:XYTUHJ}.
Here the implied constants depend on $\varepsilon$ only. 
\end{lemma}

\begin{proof}
We shall prove \eqref{Evaluate:I1} by moving the contour $\re s = \kappa$ to $\re s=\tfrac{5}{4}-\varepsilon$. 
When $\kappa\le \sigma\le \tfrac{5}{4}-\varepsilon$, it is easy to check that
$$
\min_{0\le j\le 3} (\sigma+jw_1(\sigma)-j)
= \min_{0\le j\le 3} (j+(1-j)\sigma)
\ge \tfrac{1}{2}+2\varepsilon.
$$
By Lemma \ref{Lem:Fsw} the integrand is holomorphic in the rectangle 
$\kappa\le \sigma\le \tfrac{5}{4}-\varepsilon$ and $|\tau|\le T$; 
and we can apply \eqref{UB:zeta} and \eqref{UB:Gsw} to get, in this rectangle, 
$$
\mathcal{F}_1^*(s)\ll_{\varepsilon} T^{\sigma-1} \mathcal{L}^3, 
$$
which implies that
\begin{align*}
\int_{\kappa\pm{\rm i}T}^{\frac{5}{4}-\varepsilon\pm{\rm i}T}
\frac{\mathcal{F}_1^*(s)X^{s+1}Y^{3-s}}{(2-s)(3-s)s(s+1)} \d s
& \ll_{\varepsilon} \frac{X^2Y^2\mathcal{L}^3}{T^4} 
\int_{\kappa}^{\frac{5}{4}} \bigg(\frac{XT}{Y}\bigg)^{\sigma-1} \d \sigma
\\
& \ll_{\varepsilon} \frac{X^{\frac{9}{4}}Y^{\frac{7}{4}}\mathcal{L}^2}{T^{\frac{15}{4}}}
\ll_{\varepsilon} 1.
\end{align*}
This proves \eqref{Evaluate:I1}. 

To establish \eqref{UB_TR1}, we note that for $\sigma=\frac{5}{4}-\varepsilon$ and $|\tau|\le T$ we have, as before,
$$
\mathcal{F}_1^*(s)
\ll_{\varepsilon} (|\tau|+1)^{\frac{1}{4}}, 
$$
and, by \eqref{Lem3.2_Eq_A} of Lemma \ref{Lem3.2} with $\beta=1$, 
\begin{align*}
r_{s, w_1(s)}(X, H; Y, J)
& := \big((X+H)^{s+1}-X^{s+1}\big) \big((Y+J)^{3-s}-Y^{3-s}\big)
\\
& \ll X^{\frac{5}{4}} Y^{\frac{3}{4}+\varepsilon} H J (|\tau|+1)^2. 
\end{align*}
Combining these with Lemma \ref{Operator/2}(ii), we deduce that 
\begin{align*}
(\mathscr{D}R_1)(X, X+H; Y, Y+J)
& = \frac{1}{2\pi\text{i}} \int_{\frac{5}{4}-\varepsilon-{\rm i}T}^{\frac{5}{4}-\varepsilon+{\rm i}T} 
\frac{\mathcal{F}_1^*(s)r_{s, w_1(s)}(X, H; Y, J)}{(2-s)(3-s)s(s+1)} \d s
\\
& \ll_{\varepsilon} X^{\frac{5}{4}} Y^{\frac{3}{4}+\varepsilon} H J, 
\end{align*}
from which the desired result follows. 
\end{proof}

\subsection{Evaluation of $I_2$}\

\vskip 1mm

\begin{lemma}\label{Lem:Evaluate_I2}
Under the previous notation, we have
\begin{equation}\label{Evaluate:I2}
I_2 = R_2(X, Y) + O(1)
\end{equation}
uniformly for $(X, Y, T)$ satisfying \eqref{Condition:XYTUHJ},
where 
\begin{equation}\label{def:R2XY}
R_2(X, Y)
:= \frac{4}{2\pi {\rm i}} \int_{\frac{3}{2}-{\rm i}T}^{\frac{3}{2}+{\rm i}T} 
\frac{\mathcal{F}_2^*(s)X^{s+1}Y^{(5-s)/2}}{(3-s)(5-s)s(s+1)} \d s.
\end{equation}
Further we have
\begin{equation}\label{UB_TR2}
\left.
\begin{array}{rl}
(\mathscr{D}R_2)(X, X+H; Y, Y+J)\!
\\\noalign{\vskip 1mm}
(\mathscr{D}R_2)(X-H, X; Y-J, Y)\!
\end{array}\right\}
\ll X^{\frac{3}{2}} Y^{\frac{3}{4}} HJ
\end{equation}
uniformly for $(X, Y, T, H, J)$ satisfying \eqref{Condition:XYTUHJ}.
Here the implied constants are absolute.
\end{lemma}

\begin{proof}
We shall prove \eqref{Evaluate:I2} by moving the contour $\re s = \kappa$ to $\re s=\tfrac{3}{2}$.
For $\kappa\le \sigma\le \tfrac{3}{2}$, 
we have 
$$
\min_{0\le j\le 3} (\sigma+jw_2(\sigma)-j)
= \tfrac{1}{2} \min_{0\le j\le 3} (j+(2-j)\sigma)
\ge \tfrac{3}{4}
>\tfrac{1}{2}+\varepsilon.
$$
Hence the integrand is holomorphic in the rectangle 
$\kappa\le \sigma\le \tfrac{3}{2}$ and $|\tau|\le T$, and 
we can apply \eqref{UB:zeta} and \eqref{UB:Gsw} to deduce 
$\mathcal{F}_2^*(s)\ll T^{(\sigma-1)/6}\mathcal{L}^3$
for $\kappa\le \sigma\le \tfrac{3}{2}$ and $\tau = \pm T$.
Consequently, 
\begin{align*}
\int_{\kappa\pm{\rm i}T}^{\frac{3}{2}\pm{\rm i}T} 
\frac{\mathcal{F}_2^*(s)X^{s+1}Y^{(5-s)/2}}{(3-s)(5-s)s(s+1)} \d s
& \ll \frac{X^2Y^2\mathcal{L}^3}{T^4}
\int_{\kappa}^{\frac{3}{2}} 
\bigg(\frac{X^6T}{Y^3}\bigg)^{(\sigma-1)/6} \d s
\\
& \ll \frac{X^{\frac{5}{2}}Y^{\frac{7}{4}}\mathcal{L}^2}{T^{\frac{47}{12}}} 
\ll 1, 
\end{align*}
from which \eqref{Evaluate:I2} follows. 

Next we prove \eqref{UB_TR1}.
For $\sigma=\tfrac{3}{2}$ and $|\tau|\le T$, we have, as before,
$$
\mathcal{F}_2^*(s)
\ll (|\tau|+1)^{\frac{1}{12}}\log(|\tau|+3),
$$
and, by Lemma \ref{Lem3.2} with $\beta=1$, 
\begin{align*}
r_{s, w_2(s)}(X, H; Y, J)
& := \big((X+H)^{s+1}-X^{s+1}\big)\big((Y+J)^{(5-s)/2}-Y^{(5-s)/2}\big)
\\
& \ll X^{\frac{3}{2}} Y^{\frac{3}{4}} HJ (|\tau|+1)^2. 
\end{align*}
Combining these with Lemma \ref{Operator/2}(i), we deduce that
\begin{align*}
(\mathscr{D}R_2)(X, X+H; Y, Y+J)
& = \frac{4}{2\pi\text{i}} \int_{\frac{3}{2}-{\rm i}T}^{\frac{3}{2}+{\rm i}T}  
\frac{\mathcal{F}_2^*(s) r_{s, w_2(s)}(X, H; Y, J)}{(3-s)(5-s)s(s+1)} \d s
\\
& \ll X^{\frac{3}{2}} Y^{\frac{3}{4}} HJ.
\end{align*}
This completes the proof.
\end{proof}

\subsection{Evaluation of $I_3$}\

\vskip 1mm

\begin{lemma}\label{Lem:Evaluate_I3}
Under the previous notation, for any $\varepsilon>0$ we have
\begin{equation}\label{Evaluate:I3}
I_3 = X^2Y^2P\bigg(\log X-\frac{1}{3}\log Y\bigg) + R_3(X, Y) + O_{\varepsilon}(1)
\end{equation}
uniformly for $(X, Y, T)$ satisfying \eqref{Condition:XYTUHJ},
where $P(t)$ is defined as in \eqref{def:Pt} below and
\begin{equation}\label{def:R3XY}
R_3(X, Y)
:= \frac{9}{2\pi {\rm i}} \int_{\frac{1}{2}+\varepsilon-{\rm i}T}^{\frac{1}{2}+\varepsilon+{\rm i}T} 
\frac{\mathcal{F}_3^*(s)X^{s+1}Y^{(7-s)/3}}{(4-s)(7-s)s(s+1)} \d s.
\end{equation}
Further we have
\begin{equation}\label{UB_TR3}
\left.
\begin{array}{rl}
(\mathscr{D}R_3)(X, X+H; Y, Y+J)\!
\\\noalign{\vskip 1mm}
(\mathscr{D}R_3)(X-H, X; Y-J, Y)\!
\end{array}\right\}
\ll_{\varepsilon} X^{\frac{1}{2}+\varepsilon} Y^{\frac{7}{6}} HJ
\end{equation}
uniformly for $(X, Y, T, H, J)$ satisfying \eqref{Condition:XYTUHJ}.
Here the implied constants depend on $\varepsilon$ only.
\end{lemma}

\begin{proof}
We move the line of integration $\re s = \kappa$ to $\re s=\tfrac{1}{2}+\varepsilon$.
Obviously $s=1$  
is the unique pole of order 3 of the integrand in the rectangle 
$\tfrac{1}{2}+\varepsilon\le \sigma\le \kappa$ and $|\tau|\le T$, and 
the residue is $X^2Y^2P(\log X-\tfrac{1}{3}\log Y)$ with 
\begin{equation}\label{def:Pt}
P(t)
:= \frac{1}{2!} \bigg(\frac{9(s-1)^3\mathcal{F}_3^*(s) \mathrm{e}^{t(s-1)}}{(4-s)(7-s)s(s+1)}\bigg)''\bigg|_{s=1}.
\end{equation}
When $\tfrac{1}{2}+\varepsilon\le \sigma\le \kappa$, 
we check that 
$$
\min_{0\le j\le 3} (\sigma+jw_3(\sigma)-j)
= \tfrac{1}{3} \min_{0\le j\le 3} (j+(3-j)\sigma)
\ge \tfrac{1}{2}+\varepsilon.
$$
Hence when $\tfrac{1}{2}+\varepsilon\le \sigma\le \kappa$ and $|\tau|\le T$,
\eqref{UB:zeta} and \eqref{UB:Gsw} yields 
$$
\mathcal{F}_3^*(s)\ll_{\varepsilon} T^{2(1-\sigma)/3} \mathcal{L}^3.
$$
It follows that 
\begin{align*}
\int_{\frac{1}{2}+\varepsilon\pm{\rm i}T}^{\kappa\pm{\rm i}T} 
\frac{\mathcal{F}_3^*(s)X^{s+1}Y^{(7-s)/3}}{(4-s)(7-s)s(s+1)} \d s
& \ll \frac{X^2Y^2\mathcal{L}^3}{T^4}
\int_{\frac{1}{2}}^{\kappa} 
\bigg(\frac{YT^2}{X^3}\bigg)^{(1-\sigma)/3} \d s
\\
& \ll_{\varepsilon} \frac{X^{\frac{3}{2}}Y^{\frac{11}{6}}\mathcal{L}^3}{T^{\frac{11}{3}} }
\ll_{\varepsilon} 1.
\end{align*}
These establish \eqref{Evaluate:I3}. To prove \eqref{UB_TR3}, 
we note that for $\sigma=\tfrac{1}{2}+\varepsilon$ and $|\tau|\le T$, we have 
$
\mathcal{F}_3^*(s)
\ll_{\varepsilon} (|\tau|+1)^{1/3}
$
thanks to \eqref{UB:zeta} and \eqref{UB:Gsw}, and
\begin{align*}
r_{s, w_3(s)}(X, H; Y, J)
& := \big((X+H)^{s+1}-X^{s+1}\big)\big((Y+J)^{(7-s)/3}-Y^{(7-s)/3}\big)
\\
& \ll_{\varepsilon} X^{\frac{1}{2}+\varepsilon} Y^{\frac{7}{6}} HJ (|\tau|+1)^2
\end{align*}
by Lemma \ref{Lem3.2} with $\beta=1$.
Combining these with Lemma \ref{Operator/2}(i), we deduce that
\begin{align*}
(\mathscr{D}R_3)(X, X+H; Y, Y+J)
& = \frac{9}{2\pi\text{i}} \int_{\frac{1}{2}+\varepsilon-{\rm i}T}^{\frac{1}{2}+\varepsilon+{\rm i}T}  
\frac{\mathcal{F}_3^*(s) r_{s, w_3(s)}(X, H; Y, J)}{(4-s)(7-s)s(s+1)} \d s
\\
& \ll_{\varepsilon} X^{\frac{1}{2}+\varepsilon} Y^{\frac{7}{6}} HJ.
\end{align*}
This proves the lemma. 
\end{proof} 

\subsection{Completion of proof of Theorem \ref{thm2}} 
We shall complete the proof of Theorem \ref{thm2} in this subsection. 
Denote by $\mathcal{M}(X, Y)$ the main term in the asymptotic formula 
of $M(x, y)$ in Proposition~\ref{Pro:M1XY}, that is 
$\mathcal{M}(X, Y) := X^2 Y^2 P(\psi)$ and $\psi := \log(X/Y^{1/3})$. 
Then Lemma \ref{Operator/2}(i) gives 
\begin{align*}
& (\mathscr{D}\mathcal{M})(X, X+H; Y, Y+J)
\\
& = \bigg\{XY \bigg(4P(\psi) + \frac{4}{3}P'(\psi)-\frac{1}{3}P''(\psi)\bigg)
+ O(XJ\mathcal{L}^2+YH\mathcal{L}^2)\bigg\}HJ.
\end{align*}
Since $\mathscr{D}$ is a linear operator, 
this together with Proposition~\ref{Pro:M1XY}
implies that
\begin{align*}
(\mathscr{D}M)(X, X+H; Y, Y+J)
= \bigg\{XY \bigg(4P(\psi) + \frac{4}{3}P'(\psi)-\frac{1}{3}P''(\psi)\bigg) + O_{\varepsilon}(\mathcal{R})\bigg\}HJ
\end{align*}
with
\begin{align*}
\mathcal{R}
:= X^{\frac{7}{6}+\varepsilon} Y^{\frac{11}{12}} H^{-\frac{1}{6}} 
+ X^{1+\varepsilon} Y^{\frac{13}{12}} J^{-\frac{1}{6}}
+ X^{\frac{3}{2}} Y^{\frac{3}{4}}
+ X^{\frac{1}{2}+\varepsilon} Y^{\frac{7}{6}}
+ XJ\mathcal{L}^2
+ YH\mathcal{L}^2
\end{align*}
where the terms
$X^{\frac{5}{4}+\varepsilon} Y^{\frac{3}{4}+\varepsilon}$ and $Y\mathcal{L}^2$
has been absorbed into $X^{\frac{3}{2}} Y^{\frac{3}{4}}$ and $X^{\frac{1}{2}+\varepsilon} Y^{\frac{7}{6}}$,
respectively. 
The same formula also holds for $(\mathscr{D}M)(X-H, X; Y-J, Y)$.
Now we apply Lemma~\ref{Operator/1} with 
$H=XY^{-\frac{1}{14}}$ and $J=Y^{\frac{13}{14}}$, to get 
$$
S(X, Y) 
= XY \bigg(4P(\psi) + \frac{4}{3}P'(\psi)-\frac{1}{3}P''(\psi)\bigg) 
+ O_{\varepsilon}\big(X^{\frac{3}{2}} Y^{\frac{3}{4}} 
+ X^{\frac{1}{2}+\varepsilon} Y^{\frac{7}{6}}\big),
$$
where we have used the following facts  
\begin{align*}
(X^{\frac{3}{2}} Y^{\frac{3}{4}})^{\frac{3-12\varepsilon}{5}} 
(X^{\frac{1}{2}+\varepsilon} Y^{\frac{7}{6}})^{\frac{2+12\varepsilon}{5}}
& = X^{\frac{11}{10}-\frac{(10-12\varepsilon)\varepsilon}{5}} Y^{\frac{11}{12}+\varepsilon}
\ge X^{1+\varepsilon} Y^{\frac{11}{12}+\varepsilon},
\\
(X^{\frac{3}{2}} Y^{\frac{3}{4}})^{\frac{4}{7}} 
(X^{\frac{1}{2}+\varepsilon} Y^{\frac{7}{6}})^{\frac{3}{7}}
& = X^{\frac{15}{14}} Y^{\frac{13}{14}}
\ge X^{1+\varepsilon} Y^{\frac{13}{14}}.
\end{align*}

On the other hand, a simple computation shows that $\mathscr{C}_4 = \tfrac{9}{16} \mathcal{G}(1, 1)$,
which implies immediately \eqref{def:a2}.
This finally completes the proof of Theorem~\ref{thm2}. 

\section{Proof of Theorems~\ref{thm3} and \ref{thm1}}\label{PfThm3}

\begin{proof}[Proof of Theorems~\ref{thm3}]  
The idea is to apply Theorems~\ref{thm2} in a delicate way. 
Trivially we have $r_4^*(d)\le d\tau(d)$, and therefore 
\begin{equation}\label{UB:Sxy}
S(x, y)
\le y\sum_{n\le x} \sum_{d\mid n^3} \tau(d)
\le y\sum_{n\le x} \tau(n^3)^2
\ll xy(\log x)^{15}
\end{equation}
for all $x\ge 2$ and $y\ge 2$, where the implied constant is absolute.

Let $\delta := 1-(\log B)^{-1}$ and let $k_0$ be a positive integer such that
$$
\delta^{k_0}<(\log B)^{-7}\le \delta^{k_0-1}. 
$$
Note that $k_0\asymp (\log B)\log\log B$. 
In view of \eqref{UB:Sxy}, we can write
\begin{equation}\label{UB:TB}
\begin{aligned}
T(B)
& = \sum_{1\le k\le k_0} \sum_{\delta^kB<n\le \delta^{k-1}B} \sum_{\substack{d\mid n^3\\ d<n^3/B}} r_4^*(d)
+ O(B^3)
\\
& \le \sum_{1\le k\le k_0} \sum_{\delta^kB<n\le \delta^{k-1}B} 
\sum_{\substack{d\mid n^3\\ d<\delta^{3(k-1)}B^2}} r_4^*(d)
+ O(B^3)
\\
& = \sum_{1\le k\le k_0} \big(S(\delta^{k-1}B, \delta^{3(k-1)}B^2) - S(\delta^{k}B, \delta^{3(k-1)}B^2)\big)
+ O(B^3).
\end{aligned}
\end{equation}
Similarly, 
\begin{equation}\label{LB:TB}
\begin{aligned}
T(B)
& \ge \sum_{1\le k\le k_0} \sum_{\delta^kB<n\le \delta^{k-1}B} \sum_{\substack{d\mid n^3\\ d<\delta^{3k}B^2}} r_4^*(d)
\\
& = \sum_{1\le k\le k_0} \big(S(\delta^{k-1}B, \delta^{3k}B^2) - S(\delta^{k}B, \delta^{3k}B^2)\big).
\end{aligned}
\end{equation}

On the other hand, by \eqref{Cor:Sxy} of Theorem \ref{thm2}, we have, for $1\le k\le k_0$, 
\begin{align}
S(\delta^{k-1}B, \delta^{3(k-1)}B^2) 
& = \delta^{4(k-1)} \frac{4}{9}\mathscr{C}_4B^3(\log B)^2\, 
\bigg\{1+ O\bigg(\frac{1}{\log B}\bigg)\bigg\},
\label{Proof:thm1_4}
\\\noalign{\vskip 0,5mm}
S(\delta^{k}B, \delta^{3(k-1)}B^2)
& = \delta^{4(k-1)+1} \frac{4}{9}\mathscr{C}_4 B^3(\log B)^2\,
\bigg\{1+ O\bigg(\frac{1}{\log B}\bigg)\bigg\},
\label{Proof:thm1_5}
\\\noalign{\vskip 0,5mm}
S(\delta^{k-1}B, \delta^{3k}B^2) 
& = \delta^{4k-1} \frac{4}{9}\mathscr{C}_4 B^3(\log B)^2\,
\bigg\{1+ O\bigg(\frac{1}{\log B}\bigg)\bigg\},
\label{Proof:thm1_6}
\\\noalign{\vskip 0,5mm}
S(\delta^{k}B, \delta^{3k}B^2)
& = \delta^{4k} \frac{4}{9} \mathscr{C}_4 B^3(\log B)^2\,
\bigg\{1+ O\bigg(\frac{1}{\log B}\bigg)\bigg\},
\label{Proof:thm1_7}
\end{align}
where the implied constants are absolute. 
Inserting \eqref{Proof:thm1_4} and \eqref{Proof:thm1_5} into \eqref{UB:TB}, we derive that 
\begin{align*}
T(B)
& \le (1-\delta)\frac{1-\delta^{4k_0}}{1-\delta^4} \cdot \frac{4}{9} \mathscr{C}_4 B^3(\log B)^2\,
\bigg\{1+ O\bigg(\frac{1}{\log B}\bigg)\bigg\} + O(B^3)
\\
& = \frac{1}{9} \mathscr{C}_4 B^3(\log B)^2\,
\bigg\{1+ O\bigg(\frac{1}{\log B}\bigg)\bigg\},
\end{align*}
since
$$
(1-\delta)\frac{1-\delta^{4k_0}}{1-\delta^4}
= \frac{1-\delta^{4k_0}}{1+\delta+\delta^2+\delta^3}
= \frac{1}{4} + O\bigg(\frac{1}{\log B}\bigg).
$$
Similarly, combining \eqref{Proof:thm1_6} and \eqref{Proof:thm1_7} 
with \eqref{LB:TB}, we get that 
\begin{align*}
T(B)
& \ge (\delta^{-1}-1)\frac{\delta^{4}-\delta^{4(k_0+1)}}{1-\delta^4} \cdot \frac{4}{9} 
\mathscr{C}_4 B^3(\log B)^2\,
\bigg\{1+ O\bigg(\frac{1}{\log B}\bigg)\bigg\}
\\
& = \frac{1}{9} \mathscr{C}_4 B^3(\log B)^2\,
\bigg\{1+ O\bigg(\frac{1}{\log B}\bigg)\bigg\},
\end{align*}
where have applied the estimate 
$$
(\delta^{-1}-1)\frac{\delta^{4}-\delta^{4(k_0+1)}}{1-\delta^4}
= \frac{\delta^{3}-\delta^{4k_0+3}}{1+\delta+\delta^2+\delta^3}
= \frac{1}{4} + O\bigg(\frac{1}{\log B}\bigg).
$$
The desired asymptotic formula \eqref{Evaluation:TB} follows. 
\end{proof} 

\begin{proof}[Proof of Theorem~\ref{thm1}] 
Applying \eqref{Evaluation:Sxy} of Theorem \ref{thm2} with $(x, y)=(B, B^2)$, we have 
\begin{equation}\label{Proof:thm1_3}
\sum_{n\le B} \sum_{\substack{d\mid n^3\\ d\le 4B^2}} r_4^*(d)
= \frac{4}{9} \mathscr{C}_4 B^3(\log B)^2 \, \bigg\{1+O\bigg(\frac{1}{\log B}\bigg)\bigg\}.
\end{equation}
Inserting this and \eqref{Evaluation:TB} into \eqref{def:a2},
we obtain \eqref{eq:N4*B} with $\mathcal{C}^*_4= \tfrac{16}{3}\mathscr{C}_4$.

In order to prove \eqref{eq:N4B},
we apply the inversion formula of M\"obius to write
$$
N_4(B)
= \sum_{d\le B^{1/3}} \mu(d) N_4^*\bigg(\frac{B^{\frac{1}{3}}}{d}\bigg), 
$$
where $\mu(d)$ is the M\"obius function.
Inserting \eqref{eq:N4*B} into this relation, 
we  immediately get the asymptotic formula \eqref{eq:N4B} with 
$\mathcal{C}_4 = \frac{\mathcal{C}^*_4}{9\zeta(3)}$. The theorem is proved. 
\end{proof} 

\section{General case}\label{GeneralCase: Sn}

In this section we sketch a proof of the following general result. 

\begin{theorem}\label{thm5}
Let $n$ be a positive multiple of $4$. Then as $B\to\infty$ we have 
$$
N_n(B)
= \mathcal{C}_{n}B (\log B)^2 \, \bigg\{1+O\bigg(\frac{1}{\log B}\bigg)\bigg\}
$$
and
$$
N_n^*(B)
= \mathcal{C}_n^* B^{n-1} (\log B)^2 \, \bigg\{1+O\bigg(\frac{1}{\log B}\bigg)\bigg\},
$$
where 
$$
\mathcal{C}_n := \frac{\mathcal{C}^*_n}{(n-1)^2\zeta(n-1)}, 
\qquad
\mathcal{C}^*_n
:= \frac{2n}{B_{n/2}(2^{n/2}-1)} \cdot\frac{n(n-2)}{3(3n-4)}\mathscr{C}_n, 
$$
and $\mathscr{C}_n$ is defined as in \eqref{def:Cn} below. 
\end{theorem}

Since $n=4$ has been studied, we now suppoe $n=4k$ with $k\ge 2$, 
and define by $r_n(d)$ the number of integral solutions of the equation 
\begin{equation}\label{def:rnd}
d=y_1^2+\cdots+y_n^2
\quad\text{with}\quad
(y_1, \dots, y_n)\in \Z^n. 
\end{equation}
We apply (cf. \cite[Theorem 11.2]{Iwaniec}, \cite[page 155, Theorem 1]{Grosswald1985}) to deduce, similarly 
to \eqref{def:r4} and \eqref{def:r4*pmu}, that 
\begin{equation}\label{eq:rnd}
r_n(d)
= \frac{n}{|B_{n/2}|(2^{n/2}-1)} r_n^*(d) + O(d^{n/4-1/2}),
\end{equation}
where $B_n$ is the $n$th Bernoulli number and $r_n^*(d)$ is a multiplicative function determined by
the formulae \cite[page 163, the last formula]{Grosswald1985}:
\begin{equation}\label{def:rk*pmu}
\begin{aligned}
r_n^*(p^{\mu})
& =
\begin{cases}
1 + p^{2k-1} + \cdots + p^{\mu (2k-1)} 
& \text{if $p>2$}
\\\noalign{\vskip 2mm}
(-1)^k (-1 + 2^{2k-1} + \cdots + 2^{(\mu-1)(2k-1)}) + 2^{\mu(2k-1)}
& \text{if $p=2$}
\end{cases}
\\
& =
\begin{cases}
\dfrac{1-p^{(\mu+1)(2k-1)}}{1-p^{2k-1}}  
& \text{if $p>2$}
\\\noalign{\vskip 2mm}
\bigg(1-\dfrac{(-1)^{k}}{1-2^{2k-1}}\bigg)2^{\mu(2k-1)} - (-1)^{k} \dfrac{1-2^{2k}}{1-2^{2k-1}}      
& \text{if $p=2$}
\end{cases}
\end{aligned}
\end{equation}
for all integers $\mu\ge 1$. 

It is easy to see that contribution of the error term in \eqref{eq:rnd} to $N_n(B)$ is
$$
\ll \sum_{m\le B} \sum_{\substack{d\mid m^3\\ d\le B^2}} d^{n/4-1/2}
\ll B^{n/2-1} \sum_{m\le B} \tau(m^3)
\ll B^{n/2} (\log B)^3,
$$
which is acceptable.

Define the double Dirichlet series 
$$
\mathscr{F}(s, w)
:= \sum_{m\ge 1} m^{-s} \sum_{d\mid m^3} d^{-w} r_n^*(d).
$$
By \eqref{def:rk*pmu}, we can establish the next lemma 
in the same way as before.  

\begin{lemma}\label{Lem:Fsw:general}
Let $n=4k$ with $k\ge 2$.
For $\min_{0\le j\le 3} \re (s+jw-j(2k-1))>1$, 
we have
\begin{equation}\label{Expression:Fsw:general}
\mathscr{F}(s, w) = \prod_{0\le j\le 3} \zeta(s+jw-j(2k-1)) \mathscr{G}(s, w),
\end{equation}
where $\mathscr{G}(s, w)$ is an Euler product, given by \eqref{def:Gsw:general} below.
Further, for any $\varepsilon>0$, 
$\mathscr{G}(s, w)$ converges absolutely
for $\min_{0\le j\le 3} \re (s+jw-j(2k-1))\ge \tfrac{1}{2}+\varepsilon$ and
in this half-plane
\begin{equation}\label{UB:Gsw:general}
\mathscr{G}(s, w)\ll_{\varepsilon} 1.
\end{equation}
\end{lemma}

\begin{proof}
Obviously functions $d\mapsto r_n^*(d)$ and  $m\mapsto m^{-s} \sum_{d\mid m^3} d^{-w} r_n^*(d)$ are multiplicative.
Since $r_n^*(d)\le d^{2k-1}\tau(d)$, 
for $\re s>6k-2$ and $\re w>0$ we can write the Euler product
$$
\mathscr{F}(s, w)
= \prod_p \sum_{\nu\ge 0} p^{-\nu s} \sum_{0\le \mu\le 3\nu} p^{- \mu w} r_n^*(p^{\mu})
=: \prod_p \mathscr{F}_p(s, w).
$$
Suppose $p>2$. In view of \eqref{def:rk*pmu},
we can apply \eqref{Formal_Calcul_1} with $(x, y, z) = (p^{-s}, p^{-w}, p^{2k-1})$ to write
$$
\mathscr{F}_p(s, w)
= \prod_{0\le j\le 3} \big(1-p^{-(s+jw-j(2k-1))}\big)^{-1} \mathscr{G}_p(s, w),
$$
where 
\begin{align*}
\mathscr{G}_p(s, w)
& := \bigg(
1
+ \frac{p^{2k-1}+1}{p^{s+w}}
+ \frac{p^{2(2k-1)}+p^{2k-1}+1}{p^{s+2w}}
+ \frac{p^{4k-2}+p^{2k-1}}{p^{s+3w}}
+ \frac{p^{4k-2}}{p^{2s+4w}}
\bigg)
\\
& \qquad
\times
\bigg(1-\frac{p^{2k-1}}{p^{s+w}}\bigg)
\bigg(1-\frac{p^{2(2k-1)}}{p^{s+2w}}\bigg)
\bigg(1-\frac{1}{p^{s+3w}}\bigg)^{-1}.
\end{align*}
On the other hand, a simple formal calculation shows
\begin{align*}
& 1 + \sum_{\nu\ge 1} x^{\nu} \Big(1 + \sum_{1\le \mu\le 3\nu} y^{\mu} (a z^{\mu} - b)\Big)
\\
& = \frac{1}{1-x}
+ \sum_{\nu\ge 1} x^{\nu} 
\bigg(a \frac{yz-(yz)^{3\nu+1}}{1-yz}-b\frac{y-y^{3\nu+1}}{1-y}\bigg)
\\\noalign{\vskip -0,6mm}
& = \frac{1}{1-x}
+ \frac{a}{1-yz} \bigg(\frac{xyz}{1-x} - \frac{xy^4z^4}{1-xy^3z^3}\bigg)
- \frac{b}{1-y} \bigg(\frac{xy}{1-x} - \frac{xy^4}{1-xy^3}\bigg)
\\\noalign{\vskip 0mm}
& = \frac{1}{1-x}
+ \frac{axyz(1+yz+y^2z^2)}{(1-x)(1-xy^3z^3)} 
- \frac{bxy(1+y+y^2)}{(1-x)(1-xy^3)}\cdot
\end{align*}
For $p=2$, this formula with 
$$
(x, y, z, a, b) 
= \bigg(2^{-s}, 2^{-w}, 2^{2k-1}, 
1-\dfrac{(-1)^{k}}{1-2^{2k-1}}, (-1)^{k} \dfrac{1-2^{2k}}{1-2^{2k-1}}\bigg)
$$ 
gives 
$$
\mathscr{F}_2(s, w)
= \prod_{0\le j\le 3} \big(1-2^{-(s+jw-j(2k-1))}\big)^{-1} \mathscr{G}_2(s, w),
$$
where
\begin{align*}
\mathscr{G}_2(s, w)
& := \prod_{1\le j\le 3} (1-2^{-(s+jw-j(2k-1))})
\\
& \quad
\times \bigg(1
+ a\frac{1+2^{-w+2k-1}+2^{-2w+2(2k-1)}}{2^{s+w-(2k-1)}-2^{-2w+2(2k-1)}} 
- b\frac{2^{-s-w}(1+2^{-w}+2^{-2w})}{1-2^{-s-3w}}\bigg). 
\end{align*}
These imply \eqref{Expression:Fsw:general} with
\begin{equation}\label{def:Gsw:general}
\mathscr{G}(s, w) := \prod_p \mathscr{G}_p(s, w)
\end{equation}
for $\re s>6k-2$ and $\re w>0$.

It remains to establish \eqref{UB:Gsw:general}.
We verify that
whenever $\displaystyle\min_{0\le j\le 3} (\sigma+ju-j(2k-1))\ge \tfrac{1}{2}+\varepsilon$ we have
\begin{align*}
2(\sigma+u-(2k-1))
& \ge 2(\tfrac{1}{2}+\varepsilon) = 1+2\varepsilon,
\\
2(\sigma+2u-2(2k-1))
& \ge 2(\tfrac{1}{2}+\varepsilon) = 1+2\varepsilon,
\\
\sigma+u
& \ge 2k-1+\tfrac{1}{2}+\varepsilon \ge \tfrac{7}{2}+\varepsilon,
\\
\sigma+2u
& \ge 2(2k-1)+\tfrac{1}{2}+\varepsilon \ge \tfrac{13}{2}+\varepsilon,
\\
\sigma+2u-(2k-1)
& \ge 2k-1+\tfrac{1}{2}+\varepsilon \ge \tfrac{7}{2}+\varepsilon,
\\
\sigma+3u-(2k-1)
& \ge 2(2k-1)+\tfrac{1}{2}+\varepsilon \ge \tfrac{13}{2}+\varepsilon,
\\
\sigma+3u-2(2k-1)
& \ge 2k-1+\tfrac{1}{2}+\varepsilon \ge \tfrac{7}{2}+\varepsilon,
\\
2(\sigma+2u-(2k-1))
& \ge 2(2k-1+\tfrac{1}{2}+\varepsilon) \ge 7+\varepsilon,
\\
\sigma+3u
& \ge 3(2k-1)+\tfrac{1}{2}+\varepsilon  \tfrac{13}{2}+\varepsilon.
\end{align*}
It follows that
\begin{align*}
|\mathscr{G}_p(s, w)|
& = 1 
+ O(p^{-1-\varepsilon})
\end{align*}
for $\min_{0\le j\le 3} \re(s+jw-j(2k-1))\ge \tfrac{1}{2}+\varepsilon$.
Hence the Euler product $\mathscr{G}(s, w)$ converges absolutely
for $\min_{0\le j\le 3} \re (s+jw-j(2k-1))\ge \tfrac{1}{2}+\varepsilon$, 
and \eqref{UB:Gsw:general} 
holds in this half-plane.
By analytic continuation, \eqref{Expression:Fsw:general} is also true in the same domain.
This completes the proof.
\end{proof}

Finally we give a sketch of the proof of Theorem~\ref{thm5}. 

\begin{proof}[Proof of Theorem~\ref{thm5}]  
With Lemma~\ref{Lem:Fsw} replaced by Lemma~\ref{Lem:Fsw:general}, we can establish 
Theorem \ref{thm5} in the same way as in Theorem~\ref{thm1}.  
What we have to do is just a modification of some parameters. 
For example we take, instead of \eqref{def:kappa_Lambda_L}, 
$$
\kappa := 1+\mathcal{L}^{-1},
\quad
\lambda := 2k-1+4\mathcal{L}^{-1}.
$$
Therefore instead of \eqref{def:wjs}, \eqref{def:Fk*} and \eqref{def:residue}, we have
$$
w_{j, k} = w_{j, k}(s) := (j(2k-1)+1-s)/j
\quad
(1\le j\le 3), 
$$
$$
\begin{cases}
\mathscr{F}_1^*(s)
 := \zeta(s) \zeta(2-s) \zeta(3-2s) \mathscr{G}(s, w_{1, k}(s)), 
\\\noalign{\vskip 1mm}
\mathscr{F}_2^*(s)
 := \zeta(s) \zeta(\tfrac{s+1}{2}) \zeta(\tfrac{3-s}{2}) \mathscr{G}(s, w_{2, k}(s)), 
\\\noalign{\vskip 1mm}
\mathscr{F}_3^*(s)
 := \zeta(s) \zeta(\tfrac{2s+1}{3}) \zeta(\tfrac{s+2}{3}) \mathscr{G}(s, w_{3, k}(s)), 
\end{cases}
$$
and
$$
\frac{\mathscr{F}_1^*(s)Y^{2k+1-s}}{(2k-s)(2k+1-s)}, 
\quad
\frac{4\mathscr{F}_2^*(s)Y^{(4k+1-s)/2}}{(4k-1-s)(4k+1-s)},
\quad
\frac{9\mathscr{F}_3^*(s)Y^{(6k+1-s)/3}}{(6k-2-s)(6k+1-s)}\cdot
$$
In the palce of \eqref{def:Pt} and \eqref{def:a2}, we have
$$
P_{4k}(t)
:= \frac{1}{2!} \bigg(
\frac{9(s-1)^3\mathscr{F}_3^*(s) \mathrm{e}^{t(s-1)}}{(6k-2-s)(6k+1-s)s(s+1)}\bigg)''\bigg|_{s=1}
$$
and
\begin{equation}\label{def:Cn}
\begin{aligned}
\mathscr{C}_{4k}
& := \frac{9}{16k(2k-1)} \mathscr{G}(1, 2k-1)
\\
& \;= \frac{9[(2^{2k+1}-4)(1-2^{-(6k-2)})+(-1)^k(2-2^{-2k}-2^{-4k+1}-2^{-6k+3})]}{128k(2k-1)(2^{2k-1}-1)} 
\\
& \qquad
\times
\zeta(6k-2)\prod_{p>2}\bigg(
1
+ \frac{2}{p}
+ \frac{3}{p^{2k}}
+ \frac{2}{p^{4k-1}}
+ \frac{1}{p^{4k}}
\bigg)
\bigg(1-\frac{1}{p}\bigg)^2.
\end{aligned}
\end{equation}
With the modifications above, one can establish Theorem~\ref{thm5} without 
difficulty. The details are omitted. \end{proof} 

\noindent 
{\bf Acknowledgements.} 
Zhao sincerely thanks his postdoc mentor David McKinnon for introducing him to this fantastic area of mathematics,  and for teaching him lots of algebraic geometry.  All the authors are grateful for his important contribution to this paper. We would also like to thank  Brian Conrey,  Sary Drappeau,  Lei Fu, Yuri I. Manin,  Xuanyu Pan, Per Salberger,  Mingmin Shen,  Yuri Tschinkel,   Fei Xu,  Takehiko Yasuda  and Yi Zhu  for helpful discussions.  

Liu is supported in part by NSFC grant 11531008, and Liu and Wu are supported in part by 
IRT$_{-}$16R43 from the Ministry of Education. A part of this work was done during Zhao's stay at the Max Planck Institute for Mathematics. He is very grateful for the Institute's hospitality and excellent working conditions.

\end{document}